\documentclass[11pt]{amsart}

\usepackage[margin=1in]{geometry}
\usepackage{xypic}
\usepackage{mathpple}
\usepackage{latexsym}   % use all LaTeX fonts
\usepackage{amssymb}    % use all AMS fonts
\usepackage{amsmath}    % include some AMS-LaTeX functionality
\usepackage{amsthm}     % AMS style theorem and proof environments
\usepackage{hyperref}

\numberwithin{equation}{subsection}
\newtheorem{theorem}{Theorem}[section]
\newtheorem{lemma}[theorem]{Lemma}
\newtheorem{proposition}[theorem]{Proposition}
\newtheorem{corollary}[theorem]{Corollary}
\newtheorem{application}[theorem]{Application}
\newtheorem{claim}[theorem]{Claim}

\newtheorem*{theoremtree}{Theorem \ref{thdegentree}(c)}
\newtheorem*{cornotss}{Corollary \ref{Cnotsupersing}}
\newtheorem*{theoremmono}{Theorem \ref{Tprankhypermono}/\ref{Tprankzerohypermono}}
\newtheorem*{theoremend}{Theorem \ref{Tendz}}

\theoremstyle{definition}
\newtheorem{remark}[theorem]{Remark}
\newtheorem{question}[theorem]{Question}

\newcommand{\invlim}[1]{\lim_{\stackrel{\leftarrow}{#1}}}
\newcommand{\til}[1]{{\widetilde{#1}}}

\def\calf{{\mathcal F}}
\def\calm{{\mathcal M}}

\def\shim{{\mathcal Sh}}

\newcommand{\e}[1]{{\rm\bf E}}
\def\e{{\rm\bf E}}
\def\res{{\rm\bf R}}

\newcommand{\oneover}[1]{\frac{1}{#1}}
\def\geom{{\rm geom}}
\def\et{{\text{\'et}}}
\def\del{{\partial}}

\def\mmu{{\pmb \mu}}

\newcommand{\st}[1]{\{#1\}}
\def\ra{\rightarrow}

\DeclareMathOperator{\sym}{Sym}
\global\let\hom\undefined
\DeclareMathOperator{\hom}{Hom}
\DeclareMathOperator{\spec}{Spec}
\DeclareMathOperator{\aut}{Aut}
\DeclareMathOperator{\gl}{GL}
\DeclareMathOperator{\SL}{SL}
\DeclareMathOperator{\pic}{Pic}

\DeclareMathOperator{\diag}{diag}
\DeclareMathOperator{\mat}{Mat}
\DeclareMathOperator{\jac}{Jac}
\def\sp{{\mathop{\rm Sp}}}
\DeclareMathOperator{\psp}{PSp}
\DeclareMathOperator{\gsp}{GSp}

\DeclareMathOperator{\End}{End}

\DeclareMathOperator{\gal}{Gal}
\def\tensor{\otimes}
\def\integ{\mathbb Z}
\def\gp{\mathbb G}

\def\proj{\mathbb P}
\newcommand{\abs}[1]{{\left|#1\right|}}
\newcommand{\rest}[1]{|_{#1}}
\def\rat{\mathbb Q}
\def\ff{\mathbb F}
\def\real{\mathbb R}
\def\inject{\hookrightarrow}
\def\cross{\times}
\def\units{^\cross}

\def\iso{\cong}

\def\mono{{\sf M}}

\def\cala{{\mathcal A}}

\def\calc{{\mathcal C}}
\def\calh{{\mathcal H}}

\def\calo{{\mathcal O}}

\def\calx{{\mathcal X}}
\def\inv{^{-1}}
\def\comp{\circ}
\renewcommand{\bar}[1]{{\overline{#1}}}

\newcommand{\floor}[1]{{\lfloor #1 \rfloor}}

\newenvironment{alphabetize}{\begin{enumerate}
\def\theenumi{\alph{enumi}}

}{\end{enumerate}}

\begin{document}

%\runningheads{J.\ Achter and R.\ Pries}{Monodromy of hyperelliptic $p$-rank strata}

\title{The $p$-rank strata of the moduli space of hyperelliptic curves}
%\author{Jeffrey D. Achter \& Rachel Pries}
%\author{Jeffrey D. Achter, Rachel Pries}

\author{Jeffrey D. Achter}
\email{j.achter@colostate.edu}
\address{Department of Mathematics, Colorado State University, Fort
Collins, CO 80523} 

\author{Rachel Pries}
\email{pries@math.colostate.edu}
\address{Department of Mathematics, Colorado State University, Fort
Collins, CO 80523} 
\thanks{The first author was partially supported by NSA grant H98230-08-1-0051.
  The second author was partially supported by NSF grant DMS-07-01303.}

\begin{abstract}
We prove results about the intersection of the $p$-rank strata and the
boundary of the moduli space of hyperelliptic curves in characteristic $p \geq 3$.  
This yields a strong technique that allows us to analyze the 
stratum $\calh_g^f$ of hyperelliptic curves of genus $g$ and $p$-rank $f$.
Using this, we prove that the endomorphism ring of the Jacobian of a generic hyperelliptic 
curve of genus $g$ and $p$-rank $f$ is isomorphic to $\integ$ if $g \geq 4$. 
Furthermore, we prove that the $\integ/\ell$-monodromy of
every irreducible component of $\calh_g^f$ is the symplectic
group $\sp_{2g}(\integ/\ell)$ if $g \geq 4$ or $f \geq 1$, and $\ell \not = p$ is an odd prime
(with mild hypotheses on $\ell$ when $f =0$).
These results yield numerous applications about the generic behavior of 
hyperelliptic curves of given genus and $p$-rank over finite fields, including applications 
about Newton polygons, absolutely simple Jacobians, class groups and zeta functions.\\
\end{abstract}

\maketitle

%\keywords{$p$-rank, moduli, hyperelliptic, Jacobian, monodromy}

\section{Introduction}

Suppose $C$ is a smooth connected projective hyperelliptic curve of genus $g \geq 1$ over
an algebraically closed field $k$ of characteristic $p \geq 3$.  The
Jacobian $\pic^0(C)$ is a principally polarized abelian variety of
dimension $g$.  The number of physical  $p$-torsion points of $\pic^0(C)$ is $p^f$
for some integer $f$, called the $p$-rank of $C$, with $0 \le f \le g$. 

Let $\calh_g$ be the moduli space over $k$ of smooth connected projective
hyperelliptic curves of genus $g$; it is a smooth Deligne-Mumford stack over $k$.
The $p$-rank induces a stratification $\calh_g= \cup \calh_g^f$ by
locally closed reduced substacks $\calh_g^f$,
whose geometric points correspond to hyperelliptic curves of genus $g$ and $p$-rank $f$.     

In this paper, we prove three cumulative results about $\calh_g^f$.  
The first is about the boundary of $\calh_g^f$; specifically, when $g \geq 2$, we prove that 
the boundary of every irreducible component of $\calh_g^f$ 
contains the moduli point of some singular curve which is 
a tree of elliptic curves and which has $p$-rank $f$.
The second is that the Jacobian of a generic curve of genus $g$ and
$p$-rank $f$ has endomorphism ring $\integ$ if $g \geq 4$.
The third is that, for an odd prime number $\ell$ distinct from $p$,
the mod-$\ell$ monodromy group of every irreducible component of $\calh^f_g$ is the symplectic group
$\sp_{2g}(\integ/\ell)$ if $g \geq 4$ or $f \geq 1$ (with mild hypotheses on $\ell$ when $f =0$). 
Heuristically, this means that $p$-rank constraints alone do not force the existence of
extra automorphisms (or other algebraic cycles) on a family of hyperelliptic curves.

We now state the results of this paper more precisely.

\begin{theoremtree}
Suppose $p$ is an odd prime, $g \ge 2$, and $ 0 \le f \le g$.  
Let $S$ be an irreducible component of $\calh_g^f$, the $p$-rank
$f$ stratum in $\calh_g$.  Then the closure $\bar S$ of $S$ in $\bar \calh_g$ 
contains the moduli point of some 
tree of $g$ elliptic curves, of which $f$ are ordinary and $g-f$ are
supersingular.
\end{theoremtree}

Our proof does not yield much information on the structure of the tree in Theorem \ref{thdegentree}; 
however, once the tree's structure is fixed, we prove that any choice of labeling of $f$ components as 
ordinary and $g-f$ components as supersingular will occur for some
moduli point in $\bar S$. 

Theorem \ref{thdegentree} yields a powerful technique to analyze $\calh_g^f$.  Using it, we prove the 
following two results.

\begin{theoremend}
Suppose $p$ is an odd prime, $g \geq 4$, and $0 \leq f \leq g$.  
Suppose $C$ is a generic hyperelliptic curve of genus $g$ and $p$-rank $f$ with
Jacobian $X=\pic^0(C)$.
Then $\End(X) \iso \integ$ and thus $X$ is simple. 
\end{theoremend}

The third theorem requires some notation.  Let $S$ be a connected stack
over $k$, and let $s$ be a geometric point of $S$.  Let $C \ra S$ be a
relative smooth proper curve of genus $g$ over $S$.  Then
$\pic^0(C)[\ell]$ is an \'etale cover of $S$ with geometric fibers
isomorphic to $(\integ/\ell)^{2g}$.  The fundamental group
$\pi_1(S,s)$ acts linearly on the fiber $\pic^0(C)[\ell]_s$, and the
monodromy group $\mono_\ell(C \ra S, s)$ is the image of $\pi_1(S,s)$
in $\aut(\pic^0(C)[\ell]_s)$.  For the third main result, we determine
$\mono_\ell(S):=\mono_\ell(C \to S,s)$, where $S$ is an irreducible
component of $\calh_g^f$ and $C \ra S$ is the universal curve.

\begin{theoremmono}
Suppose $p$ is an odd prime, $g \ge 1$, and $0 \le f \le g$.
Let $S$ be an irreducible component of $\calh_g^f$. 
\begin{enumerate}
\def\theenumi{\roman{enumi}}
\item  If $1 \leq f \leq g$ and $\ell$ is an odd prime distinct from $p$,
then $\mono_\ell(S ) \iso \sp_{2g}(\integ/\ell)$.
\item  If $f=0$ and if $g \ge 4$, 
then $\mono_\ell(S ) \iso \sp_{2g}(\integ/\ell)$ 
for all primes $\ell$  outside a finite set which depends only on $p$.
\end{enumerate}
\end{theoremmono}

We also prove that the $\ell$-adic monodromy group is 
$\sp_{2g}(\integ_\ell)$ in the situation of Theorem \ref{Tprankhypermono}/\ref{Tprankzerohypermono}.
(Note that the case of ordinary hyperelliptic curves, i.e., when $f=g$, follows directly from previous work, see
J.K.Yu [unpublished], \cite[Thm.\ 3.4]{AP:trielliptic}, or \cite[Thm.\ 5.1]{hall06}.)  
In addition, we determine  the $p$-adic
monodromy group of components of $\calh_g^f$ when $f \ge 1$ (Proposition \ref{proppadic}). 

This paper is a natural generalization of our paper
\cite{AP:monoprank}, which is about $\calm^f_g$, the $p$-rank strata
of the moduli space of curves.  The two papers share essential
similarities, but there are several new phenomena for hyperelliptic
$p$-rank strata which increase the difficulty of the proofs and
influence the final results in this paper.  First, the boundary
component $\Delta_0$ is more complicated for $\bar \calh_g$ than for
$\bar \calm_g$.  Second, for a singular hyperelliptic curve which is
formed as a chain of two hyperelliptic curves of smaller genera, the
set of possibilities for the location of the ordinary double point is
discrete.  These two facts play a key role in the degeneration
arguments found in the proof of Theorem \ref{thdegentree}.  
The third issue, which arises in a base case when $f=0$,
is that the stratum $\calh_3^0$ is not nearly as well understood as $\calm_3^0$.

The second and third main results of the paper 
rely on Theorem \ref{thdegentree} because the proofs use degeneration in order to proceed by induction on the genus.
For the inductive step, Theorem \ref{thdegentree} implies that 
the closure of every component $S$ of $\calh_g^f$ in $\bar
\calh_g$ intersects the stratum $\Delta_{1,1}$ of the boundary of $\calh_g$ (Corollary \ref{cornewdegen11}).  
This is used in the proof of Theorem \ref{Tendz} to show that the endomorphism ring of $\pic^0(C)$ 
acts diagonally on the Tate module.  
It is also used in the proof of Theorem \ref{Tprankhypermono}/\ref{Tprankzerohypermono} 
to show that the monodromy group of $S$ contains two non-identical
copies of $\sp_{2g-2}(\integ/\ell)$.

There are two base cases needed in this paper.  The first, when $g=2$
and $f \geq 1$, uses facts about $\calh_2^f$ from a special case of
\cite[Prop.\ 4.4]{chailadic}.  The second, when $g=3$ and $f=0$, we
found somewhat intractable.  For this reason, we employ a novel
analysis of endomorphism algebras of generic Jacobians of small genus
with $p$-rank zero.  Applying Theorem \ref{thdegentree} and
\cite[Thm.\ 1.12]{oorthess} lets us constrain the possibilities
(Lemma \ref{lemqorcubic}) for the endomorphism algebra of the Jacobian 
of a generic hyperelliptic curve of genus $3$ with $p$-rank $0$.
This, combined with an understanding of abelian varieties of Mumford
type, yields sufficient leverage to understand the case when $g=4$ and
$f=0$.  In particular, this allows us to determine the endomorphism
ring of the Jacobian of a generic hyperelliptic curve of genus $4$ and
$p$-rank $0$ and to compute the monodromy group of components of
$\calh_4^0$.

This paper also contains multiple applications about hyperelliptic curves of given genus and $p$-rank. 
For example, results on Newton polygons of (hyperelliptic) curves for arbitrary $g$ and $p$ are notoriously elusive.  
As a consequence of Theorem \ref{thdegentree}, we give an application
about Newton polygons of hyperelliptic curves.   
The application relies on and generalizes \cite[Thm.\ 1.12]{oorthess}, which is the case $g=3$.

\begin{cornotss}
Suppose $p$ is an odd prime and $g \geq 3$.  Let $S$ be an irreducible component 
of $\calh_g^0$, the $p$-rank $0$ stratum in $\calh_g$.  Then $S$ contains the moduli point
of a curve whose Jacobian is not supersingular.
\end{cornotss}

We also give applications about class groups and zeta functions
of hyperelliptic curves of given genus and $p$-rank over finite fields.
These build upon \cite[Thm.\ 9.7.13]{katzsarnak} and \cite[Thm.\ 6.1]{kowalskisieve}. 

\paragraph{{\bf Applications:}} 
Let $\ff$ be a finite field of characteristic $p$.  Under the
hypotheses of Theorem \ref{Tprankhypermono}/\ref{Tprankzerohypermono}:
\begin{enumerate}
\def\theenumi{\roman{enumi}}
\item
there is a hyperelliptic $\overline{\ff}$-curve $C$ of genus $g$ and $p$-rank $f$ 
whose Jacobian is absolutely simple
(Application \ref{appabssimp});
\item
if $|\ff|\equiv 1 \bmod \ell$, about $\ell/(\ell^2-1)$ of the hyperelliptic 
$\ff$-curves of genus $g$ and $p$-rank $f$ 
have a point of order $\ell$ on their Jacobian (Application \ref{appclass});
\item
for most hyperelliptic $\ff$-curves $C$ of genus $g$ and $p$-rank $f$, the
splitting field of the numerator of the zeta function of $C$ has
degree $2^g g!$ over $\rat$ (Application \ref{appzeta}). 
\end{enumerate}

Here is an outline of the paper.  Notation and background are found in
Section \ref{S2}.  Section \ref{Sboundresults} contains the results
about the boundary of the $p$-rank $f$ stratum $\calh_g^f$ and the
application to Newton polygons.  This section ends with some open
questions about the geometry of $\calh_g^f$.  For example, the number
of irreducible components of $\calh_g^f$ is known only in special
cases.  The results about endomorphism rings are in Section \ref{Sendresults}
while the results about monodromy and the applications to absolutely simple Jacobians,
class groups and zeta functions are in Section \ref{Smono}.

We thank the referee of \cite{AP:monoprank} for
suggestions which we used in this paper and thank F. Oort for helpful comments.

\section{Background} \label{S2}

Let $k$ be an algebraically closed field of characteristic $p \geq 3$.
With the exception of Section \ref{subsecapp}, where we work over a
finite field, all objects are defined on the category of $k$-schemes.
Let $\ell$ be an odd prime distinct from $p$.  We fix an 
isomorphism $\mmu_\ell \iso \integ/\ell$.  

\subsection{Moduli spaces}
\label{subsecmoduli}
For a natural number $g$ consider the following well-known
categories, each of which is fibered 
in groupoids over the category of $k$-schemes in the \'etale topology:
\begin{itemize}
\item{$\cala_g$} principally polarized abelian schemes of dimension $g$;
\item{$\calm_g$} smooth connected proper relative curves of genus $g$;
\item{$\calh_g$} smooth connected proper relative hyperelliptic curves of genus $g$;
\item{$\bar\calm_g$} stable curves of genus $g$. 
\end{itemize}

Each of these is a smooth Deligne-Mumford stack, and $\bar\calm_g$ is proper 
\cite[Thm.\ 5.2]{delignemumford}.   
There is a natural inclusion $\calh_g \ra \calm_g$;
let $\bar \calh_g$ be the closure of $\calh_g$ in $\bar \calm_g$.
Thus there are the following categories:
\begin{itemize}
\item{$\bar\calh_g$} stable hyperelliptic curves of genus $g$;
\item{$\til \calh_g$} stable hyperelliptic curves of genus $g$, 
along with a labeling of the smooth ramification locus (see \cite[Section 2.2]{AP:trielliptic}).
\end{itemize}
Again, these are smooth proper Deligne-Mumford stacks \cite[Thm.\ 3.2]{ekedahlhurwitz},
see e.g. \cite[Lemma 2.2]{AP:trielliptic} for $\til \calh_g$.   The forgetful map $\varpi_g:\til \calh_g \ra \bar\calh_g$ is
\'etale and Galois, with Galois group $\sym(2g+2)$.
If $S\subset \bar\calh_g$, 
let $\bar S$ be the closure of $S$ in $\bar\calh_g$.  Let
$\calc_g$ be the universal
curve over $\bar\calh_g$, i.e., the pullback to $\bar\calh_g$ of the
universal curve over $\bar\calm_g$ (cf. \cite[Thm.\ 5.2]{delignemumford}).

For a natural number $r$, let $\bar\calm_{g;r}$ be the Deligne-Mumford
stack of stable curves of genus $g$ with $r$ marked points.  Let
$\bar\calh_{g;r} =\bar\calh_g\cross_{\bar\calm_g} \bar\calm_{g;r}$ and
let $\til \calh_{g;r}= \til \calh_g \cross_{\bar\calm_g}
\bar\calm_{g;r}$.
The forgetful map $\phi_{g;r}:\bar\calh_{g;r} \ra \bar\calh_g$ is proper, flat and
surjective with connected fibers, and thus is a fibration. 
Let $\calh_{g;r} = \bar\calh_{g;r}\cross_{\bar\calh_g} \calh_g$.

\subsection{Stratifications}\label{subsecstrat}

Let $X$ be a principally polarized abelian variety of dimension $g$
defined over $k$.  There
is an integer $f(X)$ between $0$ and $g$, called the $p$-rank of $X$,
such that $X[p](k) \iso (\integ/p)^{f(X)}$.   
More generally, if $X/k$ is a semiabelian variety,
then its $p$-rank is $\dim_{\ff_p}\hom(\mmu_p,X)$.  The $p$-rank of a
curve is that of its Jacobian.  If $X \ra S$ is a
semiabelian scheme over a Deligne-Mumford stack, then there is a
stratification $S = \cup S^f$ by locally closed substacks such that $s
\in S^f(k)$ if and only if $f(X_s)=f$ 
(this follows from \cite[Thm.\ 2.3.1]{katzsf}, see, e.g., \cite[Lemma
2.1]{AP:monoprank}).  Thus, $\calh_g^f$ is the locus in $\calh_g$
parametrizing hyperelliptic curves of $p$-rank $f$.
Every component of $\calh_g^f$ has dimension $g-1+f$ \cite[Thm.\
1]{GP:05}.

Here is the definition of the Newton polygon of an abelian variety; 
see \cite[Chap.\ IV]{demazure} for details.  
The isogeny class of a $p$-divisible group $G/k$ is
determined by $\nu(G)$, a lower-convex polygon in $\real^2$ connecting
$(0,0)$ to $(\text{height}(G), \dim(G))$ with slopes $\lambda \in \rat \cap [0,1]$ and integral breakpoints.   
If $X/k$ is an abelian variety, its Newton polygon is that of its
$p$-divisible group $X[p^\infty]$.  The Newton polygon is a finer
invariant than the $p$-rank; 
indeed, the $p$-rank of $X$ is exactly the length of the slope $0$ part of the Newton polygon.
For example, $X$ is ordinary exactly when its Newton polygon only has slopes $0$ and $1$. 
By definition, $X$ is supersingular if its Newton polygon only
has slope $1/2$.

\subsection{The boundary of the moduli space of hyperelliptic curves}

The boundary of $\bar\calh_g$ is $\del\bar\calh_g = \bar\calh_g -
\calh_g$.  
The description here of $\del\bar\calh_g$ follows
\cite[Sec.\ 4(b)]{cornalbaharris} closely.  In fact, while
\cite{cornalbaharris} is written for the base field ${\mathbb C}$, the
description of $\calh_g$ and $\bar\calh_g$ is valid in any
characteristic \cite{yamaki04}.  Briefly, the irreducible components
of $\del \bar\calh_g$ come from restriction of the
components of the boundary of $\bar\calm_g$, except that $\Delta_0$
breaks into several components.   The informal discussion in this section is supplemented
with precise definitions in Section \ref{Sclutch}.

If $g \geq 2$, the boundary $\del \bar \calh_g$ is the union of
components $\Delta_i = \Delta_{i}[\bar\calh_g]$ for $1 \leq i \leq g-1$ and $\Xi_i = \Xi_i[\bar \calh_g]$ 
for $0 \leq i \leq g-2$ by \cite[p.410]{yamaki04}.  
Here $\Delta_i$ and $\Delta_{g-i}$ are the same substack of $\bar\calh_g$
and $\Xi_i$ and $\Xi_{g-i-1}$ are the same substack of $\bar \calh_g$.
Each $\Delta_i$ and $\Xi_i$ is an irreducible divisor in $\bar\calh_g$.

For $1 \leq i \leq g-1$, if $\eta$ is the generic point of $\Delta_i$, then the curve $\calc_{g,\eta}$
is a chain of two smooth irreducible hyperelliptic curves $Y_1$ and $Y_2$, of genera $i$ and $g-i$, 
intersecting in one ordinary double point $P$.  
The hyperelliptic involution $\iota$ stabilizes each of $Y_1$ and $Y_2$.
The point $P$ is a ramification point for the restriction of $\iota$
to each of $Y_1$ and $Y_2$ but is not part of the smooth ramification locus.  

If $\eta$ is the generic point of $\Xi_0$, then the curve $\calc_{g,\eta}$
is an irreducible hyperelliptic curve self-intersecting in an ordinary
double point $P$. 
The normalization of $\calc_{g,\eta}$ is a smooth hyperelliptic curve $Y_1$ of genus $g-1$
and the inverse image of $P$ in the normalization consists of an orbit under the hyperelliptic involution.

For $1 \leq i \leq g-2$, if $\eta$ is the generic point of $\Xi_i$, then the curve
$\calc_{g,\eta}$ has two components $Y_1$ and $Y_2$, which are smooth irreducible hyperelliptic 
curves, of genera $i$ and $g-1-i$, intersecting in two ordinary double points $P$ and $Q$.
The hyperelliptic involution $\iota$ stabilizes each of $Y_1$ and $Y_2$.
The points $P$ and $Q$ form an orbit of the restriction of $\iota$ to
each of $Y_1$ and $Y_2$.  
  
One can associate to a stable curve $C$ its dual graph, in which the
vertices are in 
bijection with the irreducible components of $C$ and in which 
there is an edge between two vertices exactly when the corresponding
components intersect.  A component of $C$ is called terminal if
the corresponding vertex is a leaf of the dual graph.  A curve is
called a tree if its dual graph is a tree.  
A curve is called a tree
of elliptic curves if it is a tree and if each of its irreducible
components is an elliptic curve.

A stable curve is a tree if and only if its Picard variety is
represented by an abelian scheme; such a  curve is
also said to be of compact type.
Let $\Delta_0=\Delta_0[\bar \calh_g]$ be the union of $\Xi_i$ for $0
\leq i \leq \floor{(g-1)/2}$.  The moduli points of
stable hyperelliptic curves which are not of compact type are exactly
the points of $\Delta_0[\bar \calh_g]$.

\subsection{Clutching maps} \label{Sclutch}

Recall from \cite{knudsen2} that there are 
three types of clutching maps for positive integers $g_1$ and $g_2$:
\[
\xymatrix{
\kappa_{g_1,g_2}:\til\calh_{g_1}\cross \til \calh_{g_2} \ar[r]^-{\til \kappa_{g_1,g_2}} & \til
\calh_{g_1+g_2} \ar[r]^{\varpi_{g_1+g_2}} & \bar\calh_{g_1+g_2}; \\
\kappa_{g_1}:\bar\calh_{g_1;1} \ar[r] & \bar\calh_{g_1+1};\\
\lambda_{g_1,g_2}: \bar\calh_{g_1;1}\cross \bar\calh_{g_2;1} \ar[r] &\bar\calh_{g_1+g_2+1}.
}
\]

Each of these clutching maps is the restriction of a finite, unramified morphism
between moduli spaces of labeled curves \cite[Cor.\ 3.9]{knudsen2}.
One defines: $$\Delta_{g_1}[\bar \calh_{g_1+g_2}] = {\rm Im}(\kappa_{g_1,g_2}); \ \ 
\Xi_0[\bar \calh_{g_1+1}] = {\rm Im}(\kappa_{g_1}); \ \ {\rm and} \ \  
\Xi_{g_1}[\bar \calh_{g_1+g_2+1}]= {\rm Im}(\lambda_{g_1,g_2}).$$ 
If $S$ is a stack equipped with a map $S \ra \bar\calh_g$, let
$\Delta_i[S]$ denote $S\cross_{\bar\calh_g} \Delta_i[\bar\calh_g]$. 
In particular, $\Delta_i[\til\calh_g] =
\til\calh_g\cross_{\bar\calh_g} \Delta_i$.  
Also define $\Delta_i[\bar\calh_g]^f:=(\Delta_i[\bar\calh_g])^f$. 
Similar conventions are employed for $\Xi_i$.

These clutching maps can be described in terms of their images on $T$-points for
an arbitrary $k$-scheme $T$.

\subsubsection{Information about $\kappa_{g_1, g_2}$}

For $i =1, 2$, suppose $s_i \in \til \calh_{g_i}(T)$ is the moduli
point of a hyperelliptic curve $Y_i$ with labeled smooth
ramification locus.
Then $\til\kappa_{g_1,g_2}(s_1,s_2)$ is the moduli
point of the labeled $T$-curve $Y$ where $Y$ has components $Y_1$ and
$Y_2$, and where the last ramification point of $Y_1$ and the first
ramification point of $Y_2$ are identified in an ordinary double
point.  This nodal section is dropped from the labeling of the
ramification points; the remaining smooth ramification sections of $Y_1$
maintain labels $\st{1, \ldots, 2g_1+1}$ and the remaining ramification sections
of $Y_2$ are relabeled $\st{ 2g_1+2, \ldots, 2(g_1+g_2)+2}$.  There is a unique hyperelliptic involution on
$Y$ which restricts to the hyperelliptic involution on $Y_1$ and
on $Y_2$. Moreover, $\kappa_{g_1,g_2}(s_1,s_2)$ is the moduli point of
the (unlabeled) hyperelliptic curve $Y$.

By \cite[Ex.\ 9.2.8]{blr},
\begin{align}
\label{eqblr}
\pic^0(Y) \iso \pic^0(Y_1) \cross \pic^0(Y_2).
\intertext{Then the $p$-rank of $Y$ is}
\label{eqblrprank}
f(Y) = f(Y_1)+f(Y_2).
\end{align}

\subsubsection{Information about $\kappa_{g_1}$}

Suppose $s_1 \in \bar\calh_{g_1;1}(T)$ is the moduli point of $(Y_1;P)$, a
hyperelliptic curve with a marked section.   Then $\kappa_{g_1}(s_1)$ is the moduli point of the
$T$-curve $Y$, where $Y$ is the stable model of the curve obtained by identifying  
the sections $P$ and $\iota(P)$ in an ordinary double point $P'$.  The hyperelliptic involution 
on $Y_1$ descends to a unique hyperelliptic involution on $Y$. 

By \cite[Ex.\ 9.2.8]{blr}, $\pic^0(Y)$ is an extension
\begin{equation}
\label{eqblr0}
\xymatrix{
0 \ar[r] & Z \ar[r] & \pic^0(Y)  \ar[r] & \pic^0(Y_1) \ar[r] & 0},
\end{equation}
where $Z$ is a one-dimensional torus.  
In particular, the toric rank of $\pic^0(Y)$ is one greater than that of $\pic^0(Y_1)$, 
and their maximal projective quotients are isomorphic, so that 
\begin{equation}
\label{eqblrprank0}
f(Y) = f(Y_1)+1.
\end{equation}

\subsubsection{Information about $\lambda_{g_1, g_2}$} \label{Slambda}

For $i = 1,2$, suppose $s_i \in \bar\calh_{g_i;1}(T)$ is the moduli
point of $(Y_i; P_i)$, a hyperelliptic curve with a marked section.
Then $\lambda_{g_1,g_2}(s_1,s_2)$ is the moduli point of the stable
model $Y$ of the $T$-curve obtained by identifying $P_1$ and $P_2$ in an
ordinary double point $P$ and by identifying $\iota(P_1)$ and
$\iota(P_2)$ in an ordinary double point $Q$.
There is a unique hyperelliptic involution on $Y$ that restricts to the hyperelliptic 
involution on $Y_1$ and on $Y_2$.

By \cite[Ex.\ 9.2.8]{blr}, $\pic^0(Y)$ is an extension
\begin{equation}
\label{eqblr0b}
\xymatrix{
0 \ar[r] & Z \ar[r] & \pic^0(Y)  \ar[r] & \pic^0(Y_1) \cross \pic^0(Y_2) \ar[r] & 0},
\end{equation}
where $Z$ is a one-dimensional torus.  In particular,
\begin{equation}
\label{eqblrpranklambda}
f(D) = f(Y_1)+f(Y_2)+1.
\end{equation}

For later use, here is a description of the stable model $Y$ when
$P_1$ is a ramification point of $Y_1$, but $P_2$ is not a
ramification point of $Y_2$.  Then $Y$ consists of three components,
namely the strict transforms of $Y_1$ and $Y_2$ and an exceptional
component $W$ which is a projective line.  Also $Y_1$ intersects $W$
in an ordinary double point and $Y_2$ intersects $W$ in two other
points, which are also ordinary double points.

\subsubsection{Clutching along trees}

The definition of $\til \kappa_{g_1,g_2}$ above relies on an arbitrary,
although convenient, choice of sections along which to glue, and
a labeling of the smooth ramification locus of the resulting curve.
By considering morphisms of the form $\gamma_{g_1+g_2} \comp \til
\kappa_{g_1,g_2} \comp (\gamma_{g_1} \cross \gamma_{g_2})$, where
$\gamma_{g_1+g_2} 
\in \sym(2(g_1+g_2)+2)$ and $\gamma_{g_1} \in \sym(2g_1+2)$ and
$\gamma_{g_2} \in \sym(2g_2+2)$, it is possible to clutch along
arbitrary sections, with complete control over the subsequent labeling.

This can be used to describe configurations of curves of compact type, as
follows.  A clutching tree is a finite tree $\Lambda$ along with a
choice of natural number $g_v$ for each vertex $v \in \Lambda$ such
that $\deg(v) 
\le 2g_v+2$.   Such a tree is called a clutching tree of elliptic
curves if $g_v = 1$ for all $v \in \Lambda$.
Let $\abs \Lambda$ be the number of vertices in
$\Lambda$, and let $g(\Lambda) = \sum_{v\in \Lambda} g_v$.  

Using a product of the clutching maps defined above, one
can define a morphism
\[
\xymatrix{
\kappa_\Lambda: \cross_{v\in V} \til \calh_{g_v} \ar[r]^-{\til
\kappa_{\Lambda}} & \til \calh_{g(\Lambda)} \ar[r]^{\varpi_{g(\Lambda)}}
& \bar\calh_{g(\Lambda)}.
}
\]
Let $\Delta_\Lambda$ be the image of $\kappa_\Lambda$.   If $\eta$ is
the generic point of $\Delta_\Lambda$, then $\calc_{g,\eta}$ is a
hyperelliptic curve  of compact type with dual graph
isomorphic to $\Lambda$, such that the irreducible component of
$\calc_{g,\eta}$ corresponding to the vertex $v$ has genus $g_v$.

Suppose $v_1$ and $v_2$ are adjacent vertices in a clutching tree
$\Lambda$.   Consider the tree $\Lambda'$ obtained by identifying
$v_1$ and $v_2$ in a new vertex, $v$, which is adjacent to all
neighbors of $v_1$ or $v_2$ in $\Lambda$, with label $g_v =
g_{v_1}+g_{v_2}$.  Then $\kappa_{\Lambda}$ factors through
$\kappa_{\Lambda'}$ and $\Delta_\Lambda \subset \Delta_{\Lambda'}$.  
A tree $\Lambda$ refines a tree $\Lambda'$ if $\Lambda'$ can be obtained from
$\Lambda$ through iterations of this construction.

\subsubsection{One more clutching map}
In the special case where $\Lambda$ is a tree on three vertices,
where the leaves $v_1$ and $v_3$ have $g_{v_1} = g_{v_3} = 1$, and
where $g_{v_2} = g-2$, one obtains the following diagram:

\begin{equation}
\label{diagclutchbox}
\xymatrix{
\til\calh_{1} \cross \til \calh_{g-2} \cross \til \calh_{1} 
\ar[d] \ar[r] \ar[dr]^{\til \kappa_{1,g-2,1}}
 & \til\calh_{g-1} \cross \til \calh_{1} \ar[d]
\\
\til\calh_{1}\cross \til\calh_{g-1} \ar[r] & \til \calh_g.
}
\end{equation}

Let $\Delta_{1,1}[\bar \calh_g]$ be the image of $\kappa_{1,g-2,1} = \varpi_g \comp \til \kappa_{1,g-2,1} $;
it is an irreducible component of the self-intersection locus of
$\Delta_1[\bar \calh_g]$.  If $\eta$ is the generic point of
$\Delta_{1,1}[\bar\calh_g]$, then 
the curve $\calc_{g,\eta}$ is a chain of three smooth irreducible 
hyperelliptic curves $Y_1$, $Y_2$, $Y_3$ with $g_{Y_1}=g_{Y_3}=1$ and
$g_{Y_2}=g-2$.  For $i \in \st{1,3}$, the curves $Y_i$ and $Y_2$
intersect in a point $P_i$ which is an ordinary double point.

\section{Boundary of the hyperelliptic $p$-rank strata}
\label{Sboundresults} 

\subsection{Preliminary intersection results}
\label{Sdegen}
Suppose $p\ge 3$, and $g \ge 1$ and $0 \le f \le g$.  The $p$-rank strata of the boundary 
of $\bar
\calh_g$ are easy to describe using the clutching maps.  First, if $1
\leq i \leq g-1$, then \eqref{eqblrprank}
implies that $\Delta_i[\bar \calh_g]^f$ is the union of the images of
$\til \calh_{i}^{f_1} \cross \til \calh_{g-i}^{f_2}$ under $\kappa_{i,
  g-i}$ as $(f_1,f_2)$ ranges over all pairs such that
\begin{equation}
\label{f1f2conditions}
0 \leq f_1 \leq i,\ 0 \leq f_2 \leq g-i \text{ and } f_1+f_2=f.
\end{equation}
Second, if $f \geq 1$, then
$\Xi_0[\bar \calh_g]^f$ is the image of $\bar \calh_{g-1;1}^{f-1}$
under $\kappa_{g-1}$ by \eqref{eqblrprank0}.  
Third, if $f \geq 1$ and $1 \leq i \leq g-2$, then \eqref{eqblrpranklambda}
implies that $\Xi_i[\bar \calh_g]^f$ is the image of 
$\bar \calh_{i;1}^{f_1} \cross \bar \calh_{g-1-i;1}^{f_2}$ under $\lambda_{i, g-1-i}$
as $(f_1,f_2)$ ranges over all pairs such that
\begin{equation}
\label{Xif1f2conditions}
0 \leq f_1 \leq i,\ 0 \leq f_2 \leq g-1-i \text{ and } f_1+f_2=f-1.
\end{equation} 

\begin{lemma}
\label{lemdimh}  Suppose $g \geq 2$ and $0 \le f \le g$.  
\begin{alphabetize}
\item If $1 \le i \le g-1$, then every component of
$\Delta_i[\bar\calh_g]^f$ and of $\Delta_i[\til\calh_g]^f$ has dimension $g-2+f$.
\item If $f \geq 1$ and $0 \leq i \leq g-2$, then 
every component of $\Xi_i[\bar\calh_g]^f$ and of
$\Xi_i[\til\calh_g]^f$ has dimension $g-2+f$.
\end{alphabetize}
\end{lemma}
  
\begin{proof}
For parts (a) and (b), the claims for $\bar\calh_g$ and for $\til\calh_g$ are
equivalent, since $\varpi_g$ is a finite map which preserves the
$p$-rank stratification.
Recall that $\bar\calh_g^f$ and $\til\calh_g^f$ are pure of dimension $g-1+f$ \cite[Thm.\ 1]{GP:05}.  

For part (a), suppose $0 \le f \le g$, and $1 \le i \leq g-1$, and that
$(f_1,f_2)$ is a pair which satisfies \eqref{f1f2conditions}.  Then
$\til\calh_{i}^{f_1} \cross \til \calh_{g-i}^{f_2}$ is pure of
dimension $\dim(\til \calh_{i}^{f_1}) + \dim(\til \calh_{g-i}^{f_2}) =
g-2+f$.  Since $\kappa_{i,g-i}$ is finite, $\Delta_i[\bar\calh_g]^f$
is pure of dimension $g-2+f$ as well.

For part (b), first suppose $i=0$.  
Then $\bar\calh_{g-1;1}^{f-1}$
is pure of dimension $\dim(\bar \calh_{g-1}^{f-1})+1 = g-2+f$.  Since
$\kappa_{g-1}$ is finite, $\Xi_0[\bar\calh_g]^f$ is pure of dimension
$g-2+f$ as well.

To finish part (b), suppose $1 \le i \le g-2$, and that
$(f_1,f_2)$ is a pair which satisfies \eqref{Xif1f2conditions}.  
Then $\til \calh_{i;1}^{f_1}
\cross \til \calh_{g-1-i;1}^{f_2}$ is pure of dimension 
$\dim(\til \calh_{i}^{f_1}) + \dim(\til \calh_{g-i-1}^{f_2}) +2= g-2+f$.
Since $\lambda_{i,g-1-i}$ is finite, $\Xi_i[\bar\calh_g]^f$
is pure of dimension $g-2+f$ as well.
\end{proof}

The next lemma shows that 
if $\eta$ is a generic point of $\bar\calh_g^f$, then the
curve $\calc_{g,\eta}$ is smooth. 
Thus no component of $\bar \calh_g^f$ is contained in the boundary 
$\del \bar \calh_g$.  

\begin{lemma}\label{lemhyplabeled}
Suppose $g \ge 1$ and $0 \le f \le g$.
\begin{alphabetize}
\item Then $\calh_g^f$ is open and dense in $\bar\calh_g^f$
and $\til \calh_{g}^f \cross_{\bar \calh_g} \calh_g$
is open and dense in $\til \calh_g^f$.
\item If $r \ge 1$, then $\calh_{g;r}^f$ is open and dense in $\bar \calh_{g;r}^f$
and $\til \calh_{g}^f \cross_{\bar \calh_g} \calh_g$ is open and dense in $\til \calh_g^f$.
\end{alphabetize}
\end{lemma}

\begin{proof}
Part (a) is well-known for $g=1$.
For $g \ge 2$, part (a) follows from Lemma \ref{lemdimh} since $\bar \calh_g^f$ and
$\til \calh_{g}^f$ are pure of dimension $g-1+f$ \cite[Thm.\ 1]{GP:05}.  
Part (b) follows from part (a) since the $p$-rank of a labeled curve
depends only on the underlying curve, so that $\bar\calh_{g;r}^f =
\bar\calh_{g;r} \cross_{\bar \calh_g} \bar\calh_g^f$.
\end{proof}

Suppose $S$ is an irreducible component of $\calh_g^f$ and let $\bar S$ be its closure
in $\bar \calh_g$.
Note that $\bar S$ can contain the moduli points of curves with lower
$p$-rank.  (In fact, this always happens when $f \geq 1$, see Corollary \ref{cordegenprank}.)

\begin{lemma} \label{Lintersecttheory}
\begin{alphabetize}
\item Let $S$ be an irreducible component of $\calh_g^f$.
If $\bar S$ intersects a component $\Gamma$ of $\Delta_i[\bar \calh_g]^f$
then $\bar S$ contains $\Gamma$.
\item Let $\til S$ be an irreducible component of $\til \calh_g^f$,
  with closure $S^*$.  If $S^*$ intersects a component $\til \Gamma$ of
  $\Delta_i[\til \calh_g]^f$, then $S^*$ contains $\Gamma$.
\end{alphabetize}
\end{lemma}

\begin{proof}
Part (a) is proved here; part (b) is proved in an entirely analogous fashion.
A smooth proper stack has the same intersection-theoretic properties as a 
smooth proper scheme \cite[p.\ 614]{V:stack}.
In particular, if two closed substacks of $\bar \calh_g$ intersect 
then the codimension of their intersection is at most the sum of their codimensions.
Now $\bar S$ and $\Delta_i[\bar \calh_g]$ are closed
substacks of $\bar \calh_g$. 
By Lemma \ref{lemhyplabeled}(a), $\bar S \not \subset \Delta_i[\bar \calh_g]$.
Thus the intersection of $\bar S$ with the divisor $\Delta_i[\bar\calh_g]$ has 
pure dimension $\dim\bar S - 1$,
which equals $\dim (\Delta_i[\bar\calh_g]^f)$ by Lemma \ref{lemdimh}.  
Thus if $\bar S$ intersects a component $\Gamma$ of $\Delta_i[\bar \calh_g]^f$
then it must contain the full component $\Gamma$.
\end{proof}

\begin{lemma} \label{LhypinterDelta0}
Suppose $g \ge 2$ and $0 \le f \le g$.  Let $S$ be an irreducible component of $\calh_g^f$. 
\begin{alphabetize}
\item Then $\bar S$ intersects $\Delta_0[\bar \calh_g]$ if and only if $f \ge 1$.
\item If $f \ge 1$, then each irreducible component of $\Delta_0[\bar S]$ contains
either (i) the image of a component of $\bar\calh_{g-1;1}^{f-1}$ under $\kappa_{g-1}$
or (ii) the image of a component of $\bar \calh_{i;1}^{f_1} \cross
\bar \calh_{g-1-i;1}^{f_2}$ 
under $\lambda_{i, g-1-i}$ for some $1 \leq i \leq g-2$ and some pair $(f_1,f_2)$
which satisfies \eqref{Xif1f2conditions}.  

\item If $f = 0$, then $\bar S$ contains the image of a component of
$\til \calh_i^0 \cross \til \calh_{g-i}^0$ under $\kappa_{i,g-i}$ for some $1 \le i \le g-1$. 
\end{alphabetize}
\end{lemma}

\begin{proof}
If $f = 0$, 
then \eqref{eqblrprank0} and \eqref{eqblr0b} imply that $\bar S$ does not
intersect $\Delta_0[\bar \calh_g]$.  If $f \geq 1$, then $\bar S$ is
a complete substack of dimension greater than $g-1$.  By \cite[Lemma
2.6]{FVdG:complete}, a complete substack of $\bar\calh_g - \Delta_0[\bar\calh_g]$
has dimension at most $g-1$.  Therefore, $\bar S$ intersects
$\Delta_0[\bar\calh_g]$.
This completes part (a).

For part (b), each irreducible component of $\Delta_0[\bar S]$
intersects either $\kappa_{g-1}(\bar \calh_{g-1;1}^{f-1}) \subset \Xi_0[\bar \calh_g]$ 
or $\lambda_{i,g-1-i}(\bar \calh_{i;1}^{f_1} \cross \bar \calh_{g-1-i;1}^{f_2}) \subset \Xi_i[\bar \calh_g]$ 
for some $1 \leq i \leq g-2$ and some pair $(f_1,f_2)$ which satisfies \eqref{Xif1f2conditions}.
The result then follows from Lemma \ref{Lintersecttheory}(a).
 
For part (c), recall that $\calh_g$ contains no complete
substacks of positive dimension (e.g., \cite[Cor.\ 1.9]{yamaki04}).
Thus $\bar S$ intersects $\Delta_i$ for some $0 \leq i \leq g-1$.  
By part (a), $i\not = 0$.  The result follows from Lemma
\ref{Lintersecttheory}(a).
\end{proof}

\subsection{Complements on trees}
Many of the results for $\Delta_i$ for positive $i$ have analogues
for $\Delta_\Lambda$.  For a clutching tree $\Lambda$ and a
nonnegative integer $f$, define an index set by
\begin{equation}
\label{eqflambdaconditions}
\calf(\Lambda,f) = \left\{ \left\{f_v:v \in \Lambda\right\} : 0 \le f_v \le g_v,
  \sum_v f_v = f\right\}.
\end{equation}

\begin{lemma}
\label{lemdimtree}
Let $\Lambda$ be a clutching tree with $g(\Lambda) = g$.
\begin{alphabetize}
\item The $p$-rank strata of $\Delta_\Lambda[\bar \calh_g]$ are given by
\begin{equation}
\label{eqdeltalambdaf}
\Delta_\Lambda[\bar \calh_g]^f = \bigcup_{\st{f_v} \in \calf(\Lambda,f)}
\kappa_\Lambda(\cross_{v\in \Lambda} \til \calh_{g_v}^{f_v}).
\end{equation}
\item Every component of $\Delta_\Lambda[\bar \calh_g]^f$ has dimension $g+f-\abs \Lambda$.
\end{alphabetize}
\end{lemma}

\begin{proof}
Part (a) follows from \eqref{f1f2conditions} and induction on $\abs
\Lambda$.  Part (b) follows from this and the calculation that, for
$\st{f_v} \in \calf(\Lambda,f)$,
$$
\dim(\cross_{v\in \Lambda} \til \calh_{g_v}^{f_v}) =
\sum_{v\in\Lambda}(g_v+f_v-1) = g + f - \abs \Lambda.$$
\end{proof}

\begin{lemma}
\label{lemtouchcontain}
Let $S$ be an irreducible component of $\calh_g^f$.  Let $\Lambda$ be
a clutching tree with $g(\Lambda) = g$.  If $\bar S$ intersects a
component $\Gamma$ of $\Delta_\Lambda[\bar \calh_g]^f$, then $\bar S$ contains $\Gamma$.
\end{lemma}

\begin{proof}
The proof is similar to that of Lemma \ref{Lintersecttheory}(a).
Note that
$\dim \Gamma \ge \dim
\bar S + \dim \Delta_\Lambda - \dim \bar\calh_g$.  By Lemma
\ref{lemdimtree}(b), this equals $g+f - \abs \Lambda = \dim
(\Delta_\Lambda[\bar \calh_g]^f)$.
\end{proof}

\subsection{Adjusting marked points and trees}\label{Smarked}

The next lemma shows that one can adjust the marked points of an
$r$-marked hyperelliptic curve of genus $g$ and $p$-rank $f$ without
leaving the irreducible component of $\bar\calh_{g;r}^f$ to which its
moduli point belongs.

\begin{lemma}
\label{lemlabeled}
Let $S$ be an irreducible component of $\calh_{g;r}^f$, and let $\bar
S$ be the closure of $S$ in $\bar\calh_{g;r}^f$.
Then $\bar S = \phi_{g;r}\inv(\phi_{g;r}(\bar S))$.  Equivalently, if $T$ is a
$k$-scheme, if $(C; P_1, \ldots, P_r) \in \bar S(T)$, and if $(Q_1, \ldots, Q_r)$ is
any other marking of $C$, then $(C; Q_1, \ldots, Q_r) \in \bar S(T)$.
\end{lemma}

\begin{proof}
It suffices to show that $\phi_{g;r}\inv(\phi_{g;r}(\bar S)) \subseteq \bar S$.  
Note that  $\bar S$ is the largest irreducible
substack of $\bar\calh_{g;r}^f$ which contains $S$. The fibers of
$\phi_{g;r}\rest{S}$ are irreducible, so
$\phi_{g;r}\inv(\phi_{g;r}(S))$ is also an irreducible substack of
$\bar
\calh_{g;r}^f$ which contains $S$.  Thus $\phi_{g;r}\inv(\phi_{g;r}(S))\subset
\bar S$. This shows that $\phi_{g;r}\inv(\phi_{g;r}(S)) = S$.  

To finish the proof, it suffices to show that the $T$-points of $\bar S$
and $\phi_{g;r}\inv(\phi_{g;r}(\bar S))$ coincide for an arbitrary $k$-scheme
$T$.  To this end, let $\alpha = (C; P_1, \ldots, P_r) \in \bar S(T)$,
and let $\beta = (C; Q_1, \ldots, Q_r) \in \bar\calh_{g;r}^f(T)$.
Note that $\phi_{g;r}(\beta) = \phi_{g;r}(\alpha)$, and
$\phi_{g;r}(\alpha)$ is supported in the closure of $\phi_{g;r}(S)$ in
$\bar\calh_g^f$.  
By Lemma \ref{lemhyplabeled}(b), $\calh_{g;r}$ is dense in $\bar\calh_{g;r}$.
It follows that $\beta$ is supported in the closure of $\phi_{g;r}\inv(\phi_{g;r}(S))$ in
$\bar\calh_{g;r}^f$, which is $\bar S$. 
\end{proof}

It is not clear whether one can change the labeling of the smooth ramification locus of a hyperelliptic curve
without changing the irreducible component of $\til \calh_g^f$ to which its moduli point belongs.  
To circumvent this issue, the following lemma about hyperelliptic curves of genus $2$ and $p$-rank $1$ will be useful.

\begin{lemma}
\label{lemswitchrampoint}
\begin{alphabetize}
\item First, $\bar \calh_2^1$ is irreducible and intersects $\kappa_{1,1}(\til
\calh_{1}^{1} \cross \til \calh_{1}^{0})$.
\item Second, let $\til S$ be an irreducible component of $\til\calh_2^1$.  If
$\til S$ intersects $\til\kappa_{1,1}(\til \calh_1^1 \cross \til \calh_1^0)$,
then $\til S$ also intersects $\til\kappa_{1,1}(\til \calh_1^0 \cross
\til \calh_1^1)$.
\end{alphabetize}
\end{lemma}

\begin{proof} 
For part (a), recall that the Torelli morphism $\calh_2 \ra \cala_2$ is an inclusion \cite[Lemma
1.11]{oortsteenbrink}.  Now $\dim (\calh_2^1) =\dim(\cala_2^1)$ and
$\cala_2^1$ is irreducible (e.g., \cite[Ex.\ 11.6]{evdg}).
It follows that $\calh_2^1$ is irreducible and thus $\bar \calh_2^1$ is irreducible by Lemma \ref{lemhyplabeled}(a).  Consider a chain $Y$ of
two elliptic curves, one ordinary and one supersingular, intersecting in an ordinary double point, which is a
fixed point of the hyperelliptic involution on each elliptic curve.
The moduli point of $Y$ is in the intersection of $\kappa_{1,
1}(\til \calh_{1}^{1} \cross \til \calh_{1}^{0})$ and $\bar \calh_2^1$.

For part (b), let $S^*$ be the closure of $\til S$ in $\til\calh_2$.
By hypothesis and Lemma \ref{Lintersecttheory}(b), $\til S$ contains a component of
$\til\kappa_{1,1}(\til \calh_1^1 \cross \til \calh_1^0)$.  Every
component of $\til H_1^1$ contains a component of $\til \calh_1^0$ in
its closure since any nonisotrivial proper family of curves of genus
one has supersingular fibers.  It follows that $S^*$ contains a
component of $\til \kappa_{1,1}(\til \calh_1^0 \cross \til
\calh_1^0)$, and thus intersects the closure of a component of $\til
\kappa_{1,1}(\til \calh_1^0 \cross \til \calh_1^1)$.  By Lemma
\ref{Lintersecttheory}(b), $S^*$ contains a component of $\til
\kappa_{1,1}(\til \calh_1^0\cross \til \calh_1^1)$, which then implies
the same for $\til S$.
\end{proof}

\begin{remark}
Note that a genus two curve has six ramification points and thus there are potentially up to
$6!= \abs{\sym(6)}$ irreducible components of $\til \calh_2^1$. 
In particular, the fact from Lemma \ref{lemswitchrampoint}(a) that $\bar \calh_2^1$ 
intersects $\kappa_{1,1}(\til \calh_{1}^{1} \cross \til \calh_{1}^{0})$ 
does not imply the hypothesis in part (b)
that $\til S$ intersects $\til\kappa_{1,1}(\til \calh_1^1 \cross \til \calh_1^0)$.
\end{remark}

\begin{lemma}
\label{lemswitchtree}
Let $S$ be an irreducible component of $\calh_g^f$.  Suppose $\Lambda$
is a clutching tree of elliptic curves with $g(\Lambda) = g$.
If $\bar S$ intersects
$\Delta_\Lambda[\bar \calh_g]$, then for any choice of $\st{f_v} \in \calf(\Lambda,f)$,
$\bar S$ contains an irreducible component of
$\kappa_\Lambda( \cross_{v \in \Lambda} \til \calh_{g_v}^{f_v} )$.
\end{lemma}

\begin{proof}
By Lemma \ref{lemdimtree}(a), for some choice of data
$\st{f_v^{*}} \in \calf(\Lambda,f)$,
the intersection of $\bar S$ and $\Delta_\Lambda[\bar \calh_g]$
contains a point of $\kappa_\Lambda(\cross_{v\in\Lambda} \til \calh_{g_v}^{f_v^{*}})$.
By Lemma \ref{lemtouchcontain}, there are components $\til T_v \subset
\til \calh_{g_v}^{f_v^{*}}$ such that $\bar S$ contains 
$\kappa_\Lambda( \cross_{v\in\Lambda} \til T_v)$.
One immediately reduces to the case in which 
$v_1$ and $v_2$ are adjacent vertices in $\Lambda$ with $f^*_{v_1} =
1$ and $f^*_{v_2} = 0$, and $\st{f_v} \in \calf(\Lambda,f)$ is given
by 
\[
f_v = 
\begin{cases}
f^{*}_v & v\not \in\st{v_1,v_2} \\
1 - f^{*}_v & v \in \st{v_1,v_2}.
\end{cases}
\]
Let $\Lambda'$ be the tree obtained by identifying
$v_1$ and $v_2$ in a new vertex $v_{12}$ with $g_{v_{12}} = 2$.  
By Lemma \ref{lemtouchcontain}, there is a component
$\til T_{v_{12}} \subset \til\calh_2^1$ such that $\bar S$ contains
\[
\kappa_{\Lambda'}( \til
T_{v_{12}} \cross (\cross_{v\in \Lambda', v\not = v_{12}} \til T_v) ).
\]
Now, $\til T_{v_{12}}$ contains a component of $\til\kappa_{1,1}(\til
\calh_1^{1} 
\cross \calh_1^{0})$.  By Lemma
  \ref{lemswitchrampoint}(b), $\til T_{v_{12}}$ contains a component
of $\til
  \kappa_{1,1}(\til \calh_1^{0} \cross \til \calh_1^{1})$
  as well.  Then $\bar S$ contains a component of
$\kappa_\Lambda(\cross_{v\in\Lambda} \til \calh_{g_v}^{f_v})$.
\end{proof}

\subsection{Main intersection result}

In this section, we prove that the closure of each irreducible component $S$ of $\calh_g^f$ 
contains the moduli point of a singular curve which is a tree of elliptic curves and has $p$-rank $f$.

\begin{theorem}
\label{thdegentree}
Suppose $g \ge 2$ and $0 \le f \le g$.  Let $S$ be an irreducible
component of $\calh_g^f$. 
\begin{alphabetize}
\item There exists a clutching tree of elliptic curves $\Lambda$ with
  $g(\Lambda) =  g$ such that $\bar S$ contains an irreducible
  component of  $\Delta_\Lambda[\bar \calh_g]^f$.
\item For any choice of $\st{f_v} \in \calf(\Lambda,f)$,
$\bar S$ contains an irreducible component of
$\kappa_\Lambda( \cross_{v \in \Lambda} \til \calh_{g_v}^{f_v} )$.
\item In particular, $\bar S$ contains the moduli point of some tree of
  elliptic curves, of which $f$ are ordinary and $g-f$ are supersingular.
\end{alphabetize}
\end{theorem}

\begin{proof}
It suffices to prove part (c) since parts (a) and (c) are equivalent by Lemma \ref{lemtouchcontain}
and since parts (a) and (b) are equivalent by Lemma
\ref{lemswitchtree}.

First suppose $g=2$.  If $f=2$, then $\calh_2^2$ is irreducible and affine 
and $\bar \calh_2^2$ contains the moduli point of a tree of $2$ ordinary elliptic curves.
If $f=1$ (resp.\ $f=0$), the result is true by Lemma \ref{lemswitchrampoint}(a) (resp.\ Lemma \ref{LhypinterDelta0}(c)).
Now suppose $g \geq 3$ and $0 \leq f \leq g$ and suppose as an inductive hypothesis that the result is true when 
$2 \leq g' < g$.
Let $S$ be an irreducible component of $\calh_g^f$.

\begin{claim} \label{claim}
To complete the proof, it suffices to show that $\bar S$ intersects
$\Delta_{i}[\bar \calh_g]^f$ for some $1 \leq i \le g-1$. 
\end{claim}

\begin{proof}[Proof of claim]

Suppose $\bar S$ intersects $\Delta_{i}[\bar \calh_g]^f$ for some $1
\leq i \le g-1$.  
Let $g_1=i$ and $g_2=g-i$.
By Lemma \ref{Lintersecttheory}(a), $\bar S$ contains a component of $\Delta_{g_1}[\bar \calh_g]^f$.
In other words, there exist a pair $(f_1, f_2)$ satisfying \eqref{f1f2conditions}
and, for $j=1,2$, components $\til V_j$ of $\til \calh_{g_j}^{f_j}$ 
such that $\bar S$ contains $\kappa_{g_1,g_2}(\til V_1 \cross \til V_2)$.
Then $\bar V_j= \varpi_{g_j} (\til V_j)$ is a component of $\bar \calh_{g_j}^{f_j}$.

By the inductive hypothesis, $\bar V_j$ contains the moduli point
$s_j$ of a tree $Y_j$ of $g_j$ elliptic curves, of which $f_j$ are 
ordinary and $g_j-f_j$ are supersingular.  Let $\Lambda_j$ be the dual
graph of $Y_j$.  Let $\til s_j \in \til V_j$
be such that $\varpi_{g_j}(\til s_j)=s_j$.  In other words, $\til s_j$ is
the moduli point of $Y_j$ along with the data of a choice of labeling
of the smooth ramification locus.  Then $\kappa_{g_1,g_2}(\til s_1,
\til s_2)$ is the moduli point of a curve $C$ whose dual graph is
obtained by connecting a vertex of $\Lambda_1$ with a vertex of
$\Lambda_2$.  Since $C$ is a tree, $\kappa_{g_1,g_2}(\til s_1, \til
s_2)$ is the moduli point of a tree of $g$ elliptic curves, of which
$f=f_1+f_2$ are ordinary and $g-f$ are supersingular.  This completes
the proof of the claim since $\kappa_{g_1,g_2}(\til s_1, \til s_2)$ is
in $\bar S$.
\end{proof}

Continuing the proof of Theorem \ref{thdegentree}, first suppose $f=0$.
By Lemma \ref{LhypinterDelta0}(c), $\bar S$ intersects $\Delta_i[\bar \calh_g]^f$ for
some $1 \leq i \leq g-1$.  By Claim \ref{claim}, this completes the proof when $f=0$.

Now suppose $f >0$.  By Lemma \ref{LhypinterDelta0}(a), $\bar S$ intersects $\Delta_0[\bar \calh_g]^f$.

\medskip

{\bf Case (i): $\bar S$ intersects $\Xi_0$.}

By Lemma \ref{LhypinterDelta0}(b), $\bar S$ contains the image of a
component $\bar V'$ of $\bar \calh_{g-1;1}^{f-1}$ under
$\kappa_{g-1}$.  Consider $\bar V = \phi_{g-1;1}(\bar V')$ which is a
component of $\bar \calh_{g-1}^{f-1}$.  By the inductive hypothesis,
$\bar V$ contains the moduli point of a curve $Y_1$ which is a tree of
$g-1$ elliptic curves, of which $f-1$ are ordinary and $g-f$ are
supersingular.  Let $E$ be a terminal component of $Y_1$ and let
$Y_1'$ be the closure of $Y_1-E$ in $Y_1$.  Let $R$ be the point of
intersection of $E$ and $Y_1'$.  Since the quotient of $Y_1$ by the
hyperelliptic involution $\iota$ has genus $0$, the elliptic curve $E$
is stabilized by $\iota$.  Let $P \not = R$ be a point of $E$ which is
not a ramification point of $\iota$.  By Lemma \ref{lemlabeled}, the
moduli point $t$ of $(Y_1;P)$ is in $\bar V'$.

Let $Z$ be the singular irreducible hyperelliptic curve of genus two with exactly one ordinary double point $P'$
such that the normalization of $Z$ is the elliptic curve $E$ and the pre-image of $P'$ consists of the points 
$P$ and $\iota(P)$.
In other words, the moduli point of $Z$ is the image of the moduli point of $(E;P)$ under $\kappa_1$.

Consider the point $s=\kappa_{g-1}(t)$ of $\bar S$.  
The curve $\calc_{g,s}$ has components $Z$ and $Y_1'$ which intersect in exactly one ordinary double point, $R$.
The $p$-rank of $\calc_{g,s}$ is $f(E)+f(Y_1')+1=f$. 
Since $g \geq 3$, there is a terminal component of $Y_1'$ not containing $R$ which is an elliptic curve.
Thus $s \in \Delta_1[\bar \calh_g]^f$.  (In fact, $s$ is also in $\Delta_2[\bar \calh_g]^f$ because of the component $Z$.)
By Claim \ref{claim}, this completes Case (i).

\medskip

{\bf Case (ii): $\bar S$ intersects $\Xi_i$ for some $1 \leq i \leq g-2$.}

Let $g_1=i$ and $g_2=g-1-i$.  By Lemma \ref{LhypinterDelta0}(b), $\bar
S$ contains a component of $\Xi_i[\bar \calh_g]^f$.  In other words,
there exists a pair $(f_1, f_2)$ satisfying \eqref{Xif1f2conditions}
and, for $j=1,2$, there exist components $\bar V'_j$ of $\bar
\calh_{g_j;1}^{f_j}$ such that $\bar S$ contains $\lambda_{g_1,
g_2}(\bar V'_1 \cross \bar V'_2)$.  Then $\bar V_j= \phi_{g_j;1} (\bar
V'_j)$ is a component of $\bar \calh_{g_j}^{f_j}$.

By the inductive hypothesis, $\bar V_j$ contains the moduli point
$s_j$ of a tree $Y_j$ of $g_j$ elliptic curves, of which $f_j$ are
ordinary and $g_j-f_j$ are supersingular.  Let $E_j$ be a terminal
component of $Y_j$ and let $Y_j'$ be the closure of $Y_j-E_j$ in
$Y_j$.  Let $R_j$ be the point of intersection of $E_j$ and $Y_j'$.
Since the quotient of $Y_j$ by the hyperelliptic involution $\iota_j$
has genus $0$, the elliptic curve $E_j$ is stabilized by $\iota_j$.
Let $P_1 \not = R_1$ be a point of $E_1$ which is a ramification point
of $\iota_1$.  Let $P_2 \not = R_2$ be a point of $E_2$ which is not a
ramification point of $\iota_2$.

By Lemma \ref{lemlabeled}, the moduli point $s'_j$ of $(Y_j;P_j)$ is
in $\bar V'_j$.  Consider $s=\lambda_{g_1, g_2}(s'_1,s'_2)$ which is a
point of $\bar S$.  By Section \ref{Slambda}, the components of the
stable model $\calc_{g,s}$ are the strict transforms of $Y_1$ and $Y_2$ and an
exceptional component $W$ which is a projective
line.  Moreover, $Y_1$ 
intersects $W$ in an ordinary double point and $Y_2$ intersects $W$ in
two other points, which are also ordinary double points.  The $p$-rank
of $\calc_{g,s}$ is $f(Y_1)+f(Y_2)+1=f$.

The curve $\calc_{g,s}$ has a terminal component $E'_1$ of genus $1$.
To see this, when $i=1$ let $E'_1=E_1$, and when
$i > 1$ let $E'_1 \not = E_1$ be another terminal component of $Y_1$.  
It follows that $s$ is in $\Delta_1[\bar \calh_g]^f$.  
(Also $s$ is in $\Delta_{i}[\calh_g]^f$ because of the component $Y_1$.)
By Claim \ref{claim}, this completes Case (ii).
\end{proof}

\subsection{Three corollaries}

Here are several consequences of Theorem \ref{thdegentree} which will
be used later in the paper.  

In the setting of Theorem \ref{thdegentree}, one can deduce that
$\bar S$ intersects $\Delta_i$ nontrivially only when $\Lambda$ has an
edge whose removal yields two trees of size $i$ and
$g-i$.  This is only guaranteed when $i=1$.
Luckily, the following information on degeneration to $\Delta_1$
is sufficient for the later applications in the paper.

\begin{corollary}
\label{cornewdegen1}
Suppose $g \ge 2$ and $0 \le f \le g$.  Let $S$ be an irreducible
component of $\calh_g^f$.  Then $\bar S$ intersects $\Delta_1[\bar \calh_g]^f$.  Furthermore:
\begin{alphabetize}
\item if $f \le g-1$, then $\bar S$ contains an irreducible component
  of $\kappa_{1,g-1}(\til \calh_1^0 \cross \til \calh_{g-1}^f)$; and
\item if $f \ge 1$, then $\bar S$ contains an irreducible component of
  $\kappa_{1,g-1}(\til \calh_1^1 \cross \til \calh_{g-1}^{f-1})$.
\end{alphabetize}
\end{corollary}

\begin{proof}
By Theorem \ref{thdegentree}(c), $\bar S$ contains the moduli point of a tree of
elliptic curves, of which $f$ are ordinary and $g-f$ are supersingular.
Every tree has a leaf; by Theorem \ref{thdegentree}(b), that leaf
can be chosen to be ordinary or supersingular if the obvious necessary
constraint is satisfied.  The result follows by Lemma \ref{lemtouchcontain}.
\end{proof}

The $\ell$-adic and $p$-adic monodromy proofs in Section \ref{Smono} rely on degeneration
to $\Delta_{1,1}$.  One can label the four possibilities for
$(f_1,f_2,f_3)$ such that $f_1 +f_2+f_3=f$ and $0 \le f_1,f_3 \le 1$
as follows:
(A) $(1,f-2,1)$;
(B) $(0,f-1,1)$; 
(B') $(1,f-1,0)$; and
(C) $(0,f,0)$.

\begin{corollary} 
\label{cornewdegen11}
Suppose $g \geq 3$ and $0 \leq f \leq g$.  Let $S$ be an irreducible component of $\calh_g^f$.  
\begin{alphabetize}
\item Then $\bar S$ intersects $\Delta_{1,1}[\bar \calh_g]^f$.
\item There is an irreducible component $\til S$ of $\bar S \cross
  \til \calh_g$, and a choice of $(f_1,f_2,f_3)$ from cases (A)-(C); 
and there are irreducible components $S_1$ of $\til \calh_{1}^{f_1}$
and $S_2$ of $\til\calh_{g-2}^{f_2}$ 
and $S_3$ of $\til \calh_{1}^{f_3}$;
and there are irreducible components $S_R$ of $\til \calh_{g-1}^{f_2+f_3}$ and $S_L$ of 
$\til \calh_{g-1}^{f_1+f_2}$;
such that the restriction of the clutching maps of \eqref{diagclutchbox} yields a commutative diagram
\begin{equation}
\label{diagclutch}
\xymatrix{
S_1 \cross S_2 \cross S_3 \ar[r]  \ar[d] \ar[dr]^{\til \kappa_{1,g-2,1}}& 
 S_1 \cross S_R \ar[d]\\
 S_L \cross S_3 \ar[r] & \til S \cap \Delta_{1,1}[\til \calh_g].
}
\end{equation}
\item Furthermore, case (A) occurs as long as $f \geq 2$, cases (B) and (B') occur as long as $1 \leq f \leq g-1$, 
and case (C) occurs as long as $f \leq g-2$. 
\end{alphabetize}
\end{corollary}

\begin{proof}
By Theorem \ref{thdegentree}(a), there is a clutching tree of elliptic
curves $\Lambda$ such that $\bar S$ contains a component of
$\Delta_\Lambda[\bar \calh_g]^f$.  Let $v_1$ and $v_3$ be two leaves of $\Lambda$; using
Theorem \ref{thdegentree}(b), one can assume that $f_{v_i} = f_i$ where $(f_1,f_2,f_3)$ is chosen as in part (c).  
Let $\Lambda'$ be the tree obtained by coalescing all vertices of
$\Lambda$ except for $v_1$ and $v_3$.  Let $v_2$ denote this new vertex
and let $f_{v_2} = f_2$.  Since $\Lambda$ refines $\Lambda'$, 
then $\bar S$ intersects $\Delta_\Lambda[\bar \calh_g]^f$ which completes part (a).
Moreover, there is an irreducible component $\til S$ of $\bar S \cross_{\bar \calh_g} \til \calh_g$ such that
$\til S$ intersects $\Delta_{1,1}[\til \calh_g]^f$.
Part (b) follows from the definition of $\Delta_{1,1}[\til \calh_g]^f$ and Lemma \ref{lemtouchcontain}.
\end{proof}

\begin{corollary}
\label{cordegenprank}
Let $S$ be an irreducible component of $\calh_g^f$.  For each $0 \le
f' < f$, there exists an irreducible component $T$ of $\calh_g^{f'}$
such that $\bar S$ contains $T$.
\end{corollary}

\begin{proof}
It suffices to prove the result for $f' =
f-1$.  Let $S^*$ be the closure of $S$ in $\bar\calh_g -
\Delta_0[\bar\calh_g]$.  A purity result \cite[Lemma
1.6]{O:purity} shows that $S^* - (S^*)^f$, if nonempty, is pure of
dimension $\dim S^* - 1$.  In particular, let $Z = (S^*)^{f-1}$; then
$Z$, if nonempty, is pure of dimension $g-2+f$.  

By Corollary
\ref{cornewdegen1}(b), $S^*$ contains an irreducible component of
$\kappa_{1,1}(\til \calh_1^1 \cross \til \calh_{g-1}^{f-1})$.
Since $\til \calh_1^1$ is dense in $\til \calh_1$, its closure
contains the moduli points of supersingular elliptic curves (with
labeled smooth ramification locus).  Therefore, $S^*$ contains an
irreducible component of $\kappa_{1,1}(\til \calh_1^0 \cross \til
\calh_{g-1}^{f-1})$, and $Z$ is nonempty.  Then $\dim Z = g-2+f = \dim
\bar \calh_g^{f-1}$, and so $Z$ contains a component $\bar T$ of
$\bar\calh_g^{f-1}$.  By Lemma \ref{lemhyplabeled}(a) $\bar
S$ contains a component $T$ of $\calh_g^{f-1}$.
\end{proof}

\subsection{Application to Newton polygons}

Recall that a stable curve $C$ of compact type is supersingular if
all the slopes of the Newton polygon of its Jacobian equal $1/2$.
This is equivalent to the condition that the Jacobian of $C$ is isogenous to 
a product of supersingular elliptic curves.
Note that a supersingular curve necessarily has $p$-rank zero.  
An abelian variety of $p$-rank zero is necessarily supersingular only when the dimension satisfies $g \leq 2$.

In this section, we prove that the Newton polygon of a generic hyperelliptic curve of $p$-rank $0$ 
is not supersingular when $g \geq 3$.
The result generalizes \cite[Thm.\ 1.12]{oorthess} which is the case $g=3$.

Newton polygons have the following semicontinuity property: let $S =
\spec(R)$ be the spectrum of a local ring, 
with generic point $\eta$ and geometric closed point $s$;
if $G$ is a $p$-divisible group over $S$, then $\nu(G_{\eta})$ either
equals or lies below $\nu(G_{s})$.  
(The latter condition means that 
$\nu(G_{\eta})$ and $\nu(G_{s})$ have the same endpoints and
all points of $\nu(G_{\eta})$ lie below $\nu(G_{s})$.) 

\begin{corollary} \label{Cnotsupersing}
Suppose $p$ is an odd prime and $g \ge 3$. 
Let $\eta$ be a generic point of $\calh_g^0$.  Then $\calc_{g,\eta}$ is not supersingular.
In particular, there exists a smooth hyperelliptic curve of genus $g$ and $p$-rank $0$ which is not supersingular.
\end{corollary}

\begin{proof}
When $g =3$, this follows from \cite[Thm.\ 1.12]{oorthess}.  
For $g \geq 4$, the proof proceeds by induction.
Let $S$ be the closure of $\eta$ in $\calh_g^0$.  
By Corollary \ref{cornewdegen1}, $\bar S$ contains a component of $\Delta_1[\bar \calh_g]^0$.
Thus there are components $\til V_1$ of $\til \calh_1^0$ and $\til V_2$ of $\til \calh_{g-1}^0$
such that $\bar S$ contains $\kappa_{1,g-1}(\til V_1 \cross \til V_2)$.
By the inductive hypothesis, the Newton polygon of the generic point of $\til V_2$ is not supersingular;
in particular, it has a slope $\lambda$ such that $0 < \lambda < 1/2$.  
The same is then true of the generic point of $\kappa_{1,g-1}(\til V_1 \cross \til V_2)$.
By semicontinuity \cite[Thm.\ 2.3.1]{katzsf}, the generic Newton polygon of $\bar S$ (and thus of $S$)
either equals or lies below 
that of $\kappa_{1,g-1}(\til V_1 \cross \til V_2)$.  In particular, it has a slope $\lambda' <1/2$. 
Thus $\calc_{g, \eta}$ is not supersingular.
\end{proof}

\begin{remark}
When $p=2$ (a case not considered in this paper), there are some
results about the slopes of Newton polygons of hyperelliptic curves of
$p$-rank $0$,  
see e.g.\ \cite{Zhu:nohypss}.
\end{remark}

\subsection{Open questions about the geometry of the hyperelliptic $p$-rank strata} \label{Squestions}

\begin{question} \label{Q0} 
Does the closure of each component of $\calh_g^f$ contain the moduli point of a chain of elliptic curves with $p$-rank $f$?  
\end{question}

If the answer to Question \ref{Q0} is affirmative then Lemma \ref{lemswitchtree} 
implies that every ordering of $f$ ordinary and $g-f$ supersingular elliptic curves occurs for such a chain. 
In \cite[Cor.\ 3.6]{AP:monoprank}, the authors show the analogous question has a positive answer 
for every component of $\calm_g^f$.  The difference for $\calm_g^f$ is that the clutching morphism identifies two curves
at an arbitrary point of each, rather than a ramification point of
each.  The location of these points can 
then be changed using an analogue of Lemma \ref{lemlabeled}.  

\begin{question}\label{Q3}
For $2 \leq i \leq g-2$, does the closure of each component of $\calh_g^f$ intersect $\Delta_i[\calh_g]^f$?
\end{question}

In \cite[Prop.\ 3.4]{AP:monoprank}, the authors show that the analogous question has a positive answer 
for every component of $\calm_g^f$, also with control over the arrangement of $p$-ranks.  
An affirmative answer to Question \ref{Q0} would imply an affirmative answer to Question \ref{Q3}.

\begin{question} \label{Q1}
How many irreducible components does $\calh^f_g$ have?
\end{question}

One knows that $\calh^f_g$ is irreducible for all $p$ when $f=g$ or when $g=2$ and $f=1$.
If $g \geq 3$, then $\cala_g^f$ is irreducible by \cite[Remark 4.7]{chailadic}.
If $\calh^f_g$ is irreducible, then there is a very short proof that Questions \ref{Q0} 
and \ref{Q3} have affirmative answers.

%\begin{question} \label{Qnewton}
%Let $\eta$ be a generic point of $\calh_g^f$.  What is the Newton polygon of the Jacobian of $\calc_{g,\eta}$?
%\end{question}

%The answer to Question \ref{Qnewton} is known only for pairs $(g,f)$ such that $g-f \leq 2$.
%For the generic point $\xi$ of $\cala_g^0$ with $g\ge 3$, 
%one knows that the Newton polygon of the abelian variety $X_\xi$ is 
%$\st{1/g, (g-1)/g}$ \cite[Cor.\ 3.2]{oorttexelnp}.

\section{Endomorphism rings} \label{Sendresults}

In this section, we use degeneration results from Section \ref{Sboundresults} to constrain the
endomorphism ring of a generic curve of given genus and $p$-rank.

Let $\calx_g = \pic^0_{\calc_g/\bar\calh_g}$ be the neutral component
of the relative Picard functor of $\calc_g$ over $\bar\calh_g$; then
$\calx_g \ra \bar\calh_g$ is a semiabelian scheme.  To ease notation,
if $X$ is an abelian variety, let $\e(X) = \End(X)\tensor\rat$, and
let $\e_\ell(X) = \e(X)\tensor_\rat \rat_\ell$; then $\e_\ell(X)$ acts
the rational Tate module $V_\ell(X) :=
T_\ell(X)\tensor_{\integ_\ell}\rat_\ell$.  If $X$ is simple, then the
center of $\e(X)$ is either a totally real or totally imaginary number
field.

\begin{lemma}
\label{lemqorcubic}
Let $\xi$ be a geometric generic point of $\calh_3^0$.  
Then $\calx_{3,\xi}$ is simple and either $\e(\calx_{3,\xi})\iso \rat$
or $\e(\calx_{3,\xi})$ is isomorphic to a totally real cubic field $L$
such that $L\tensor\rat_p$ is a field.
\end{lemma}

\begin{proof}
Suppose there is an isogeny $\calx_{3,\xi} \sim A_1 \oplus A_2$ for abelian varieties $A_1$ and $A_2$ of dimensions 
$1$ and $2$.  Then $A_1$ and $A_2$ each have $p$-rank $0$ and are thus supersingular.
Then $\calx_{3,\xi}$ is supersingular, which contradicts the fact that 
the Newton polygon of $\calx_{3,\xi}$ has slopes $1/3$ and $2/3$ \cite[Thm.\ 1.12]{oorthess}.
Thus $\calx_{3,\xi}$ is simple.

By the classification of 
endomorphism algebras of simple abelian varieties of prime dimension (e.g.,
\cite[7.2]{oortendabvar}), to complete the proof it suffices to show that neither a complex
multiplication field of degree six nor a quadratic imaginary field
acts on $\calx_{3,\xi}$.
Let $S$ be the closure of $\xi$ in $\calh_3^0$.  Since $\dim S = 2 >
0$ but abelian varieties with complex multiplication are rigid, $\e(\calx_{3,\xi})$ is not a
complex multiplication field of degree $6$. 

To address the possibility of an action by a quadratic imaginary
field $K$, suppose to the contrary that there is a subring of $\End(\calx_{3,\xi})$
isomorphic to an order $\calo_K$ in $K$.  Then $\calo_K\tensor\integ_p$ acts on the $p$-divisible group
$\calx_{3,\xi}[p^\infty]$.  There is an inclusion $K\tensor\rat_p
\inject \End(\calx_{3,\xi}[p^\infty])\tensor \rat_p \iso D_{1/3} \oplus D_{2/3}$.
(Here, $D_\lambda$ denotes the central simple $\rat_p$-algebra with
Brauer invariant $\lambda$.)  Every maximal subfield of $D_{1/3}$
or $D_{2/3}$ is a cubic extension of $\rat_p$, but $K\tensor\rat_p$ is
a $\rat_p$-algebra of degree two, so $K \tensor \rat_p$ cannot be a field.  In particular, $\calx_{3,\xi}$ does not admit an action by a
quadratic imaginary field inert or ramified at $p$.

Finally, suppose $\calx_{3,\xi}$ admits an action by a quadratic imaginary field $K$
which splits at $p$.  Let $(r,s)$ be the signature of the action of
$\calo_K$ on ${\rm Lie}(\calx_{3,\xi})$; the dimensions $r$ and $s$ are nonnegative and $r+s =
3$.  Consider the moduli space $\shim_{\calo_K;(r,s)}$ of
abelian threefolds with an action by $\calo_K$ of signature $(r,s)$.
The Torelli morphism $\tau$ restricts to a finite morphism from
$S$ to a component of $\shim_{\calo_K;(r,s)}$.
Since $\dim S = 2$ and $\dim \shim_{\calo_K;(r,s)} = r \cdot s$,
then $(r,s)$ is either $(1,2)$ or $(2,1)$.
Thus $\tau(S)$ is dense in $\shim_{\calo_K;(r,s)}$.  
This gives a contradiction since $\calx_{3,\xi}$ has $p$-rank zero but the generic
member of $\shim_{\calo_K;(r,s)}$ is ordinary \cite[Thm.\
1.6.2]{wedhorn}.

Therefore, $\e(\calx_{3,\xi})$ is either $\rat$ or a totally real cubic
field, $L$.  In the latter case, the $p$-rank zero locus in a Hilbert
modular threefold attached to $L$ has dimension (at least, and thus
equal to) two, which forces $L\tensor\rat_p$ to be a field.
\end{proof}

\begin{remark}
In the situation of Lemma \ref{lemqorcubic}, it is not known which
outcomes occur.
\end{remark}

\begin{lemma}
\label{lemendnotcubic}
Let $Y$ be a simple abelian variety whose dimension $g$ is relatively prime to $3$.  
If there exists a geometric generic point $\xi_3$ of $\calh_3^0$ for which there is a nontrivial homomorphism 
$\End(Y) \ra \e(\calx_{3,\xi_3})$, then $\End(Y) \iso \integ$.
\end{lemma}

\begin{proof}
If $Z$ is a simple abelian variety, let $\e_0(Z)$ be the subfield of
$\e(Z)$ fixed by the Rosati involution, and let $e_0(Z) =
[\e_0(Z):\rat]$.  Then $e_0(Z) | \dim(Z)$.

By Lemma \ref{lemqorcubic}, $\e(\calx_{3,\xi_3})$ is a totally real
field of dimension $1$  
or $3$ over $\rat$.  On one hand, the existence of a nontrivial homomorphism
$\End(Y) \ra \e(\calx_{3,\xi_3})$ forces $e_0(Y)$ to divide
$e_0(\calx_{3,\xi_3})$, and thus 
$e_0(Y)|3$.  On the other hand, $e_0(Y)|g$.  Therefore, $e_0(Y) =
1$ and $\e_0(Y) \iso \rat$.  Neither a noncommutative algebra nor a
totally imaginary 
field admits a nontrivial homomorphism to $\e(\calx_{3,\xi_3})$, and thus
$\e(Y) = \e_0(Y)$ and  $\End(Y) \iso \integ$.
\end{proof}

\begin{lemma}
\label{lemendextend}
Let $X \ra S$ be a polarized abelian scheme over a reduced irreducible
Noetherian stack.  Let $\eta$ be the generic point of $S$, and let
$s\in S$ be any point.  Then there exists an inclusion $\End(X_{\bar\eta})
\inject \End(X_{\bar s})$.
\end{lemma}

\begin{proof}
By introducing a rigidifying structure on $X\ra S$, such as coordinates on
the space of sections of the third power of the ample line
bundle given by the polarization, one can assume $S$ is a reduced irreducible Noetherian scheme.
Since the absolute endomorphism ring of an abelian variety is defined
over a finite extension of the base field, it suffices to show the existence of
an inclusion $\End(X_\eta) \inject \End(X_s)$.  If $S$ is normal, then
$\End(X_\eta)$ extends uniquely to $\End(X_S)$, and in particular to
$\End(X_s)$ \cite[I.2.7]{faltingschai}.  In general, let $\nu: S' \ra
S$ be the normalization map; let $\eta'$ be the generic point of $S'$,
and let $s'$ be a point of $S'$ over $s'$.  The desired result follows
from the canonical map $\End((\nu^*X)_{\eta'}) \inject
\End((\nu^*X)_{s'})$ and the isomorphisms of abelian varieties
$(\nu^*X)_{\eta'} \iso X_\eta\cross \eta'$ and $(\nu^*X)_{s'} \iso
X_s\cross s'$.
\end{proof}

\begin{proposition}
\label{Pendz4}
If $\xi_4$ is a geometric generic point of $\calh_4^0$, 
then $\calx_{4, \xi_4}$ is simple and $\End(\calx_{4,\xi_4}) \iso \integ$. 
\end{proposition}

\begin{proof}
Let $S_4$ be the closure of $\xi_4$ in $\calh_4^0$.
Suppose there is an isogeny $\calx_{4,\xi_4} \sim A_1 \oplus A_2$ for two abelian varieties $A_1$ and $A_2$.
If $A_1$ and $A_2$ each have dimension $2$, then they are
supersingular since they have $p$-rank $0$. 
Then $\calx_{4,\xi_4}$ is supersingular, which contradicts Corollary \ref{Cnotsupersing}. 
If $A_1$ has dimension $1$ and $A_2$ has dimension $3$, then 
there is a curve $W$ of genus $3$ such that $\jac(W) \iso A_2$.
The inclusion of $A_2$ into $\calx_{4,\xi_4}$ yields a cover $\psi:\calc_{4,\xi_4} \to W$.
By the Riemann-Hurwitz formula $6 \geq 4{\rm deg}(\psi)$ which is impossible since ${\rm deg}(\psi) \geq 2$. 
Thus $\calx_{4,\xi_4}$ is simple.

By Corollary \ref{cornewdegen1}(a) there exist components $\til V_1 \subset
\til \calh_1^0$ and $\til V_2 \subset \til\calh_3^0$ such that $\bar S_4$
contains $\kappa_{1,3}(\til V_1 \cross \til V_2)$.  
Let $\xi_1$ and $\xi_3$ be geometric generic points of $\til V_1$ and
$\til V_2$, respectively, and let $\eta = \kappa_{1,3}(\xi_1,\xi_3)$.  
Since $\calx_{3,\xi_3}$ is simple by Lemma \ref{lemqorcubic}, there are no nontrivial homomorphisms between
$\calx_{3,\xi_3}$ and $\calx_{1,\xi_1}$.  This yields an isomorphism 
\[
\e(\calx_{4,\eta}) \iso \e(\calx_{1,\xi_1}) \oplus
\e(\calx_{3,\xi_3}).
\]
By Lemma \ref{lemendextend}, there is an inclusion 
$\e(\calx_{4,\xi_4}) \inject \e(\calx_{4, \eta})$ and thus an inclusion
$\e(\calx_{4,\xi_4}) \inject \e(\calx_{3, \xi_3})$.
Since $\calx_{4, \xi_4}$ is simple, Lemma \ref{lemendnotcubic} implies that $\End(\calx_{4,\xi_4}) \iso \integ$.
\end{proof}

\begin{theorem}
\label{Tendz}
Suppose $g \geq 4$ and $0 \leq f \leq g$.  
If $\xi$ is a geometric generic point of $\calh_g^f$,
then $\End(\calx_{g,\xi}) \iso \integ$ and thus $\calx_{g, \xi}$ is simple. 
\end{theorem}

\begin{proof}
By Corollary \ref{cordegenprank} and Lemma \ref{lemendextend}, it suffices to prove the result when $f=0$.
For $f=0$, the proof is by induction on $g$ with the base case $g=4$ supplied by Proposition \ref{Pendz4}.
Suppose $g \geq 5$ and let $S$ be the closure of $\xi$ in $\calh_g^0$.
By Corollary \ref{cornewdegen11}, $\bar S$ intersects $\Delta_{1,1}[\bar \calh_g]^0$ 
and there is an irreducible component $\til S$ of $\bar S \cross
\til \calh_g$, and there are irreducible components $S_1$ of $\til \calh_{1}^{0}$
and $S_2$ of $\til\calh_{g-2}^{0}$ 
and $S_3$ of $\til \calh_{1}^{0}$;
and there are irreducible components $S_R$ of $\til \calh_{g-1}^{0}$ and $S_L$ of 
$\til \calh_{g-1}^{0}$;
such that the restriction of the clutching maps yields a commutative diagram
\begin{equation}
\label{diagclutchg5}
\xymatrix{
S_1 \cross S_2 \cross S_3 \ar[r]  \ar[d] \ar[dr]^{\til \kappa_{1,g-2,1}}& 
 S_1 \cross S_R \ar[d]\\
 S_L \cross S_3 \ar[r] & \til S \cap \Delta_{1,1}[\til \calh_g].
}
\end{equation}
Let $\eta_i$ be the generic point of $S_i$ for $1 \leq i \leq 3$;
similarly, let $\eta_L$ be the generic point of $S_L$, and $\eta_R$ that
of $S_R$.  Let $s = \til \kappa_{1,g-2,1}(\eta_1,\eta_2,\eta_3)$.
By Lemma \ref{lemendextend}, there are inclusions
\begin{equation}
\label{diagend}
\xymatrix{
\End(\calx_{g,s})  & \ar[l] \End(\calx_{g, \til \kappa_{1,g-1}(\eta_1 \cross \eta_R)})\\
\End(\calx_{g, \til \kappa_{g-1,1}(\eta_L \cross \eta_{3})}) \ar[u] & \ar[l] \ar[u] \End(\calx_{g, \xi}).
}
\end{equation}

There is a canonical isomorphism of rational Tate modules
$V_\ell(\calx_{g,s}) \iso V_\ell(\calx_{1,\eta_1}) \cross
V_\ell(\calx_{g-2,\eta_2}) \cross V_\ell(\calx_{1,\eta_3})$.  Choose
coordinates on $V_\ell(\calx_{g,s})$ 
compatible with this decomposition.  
On one hand, by the inductive hypothesis,
$\e_\ell(\calx_{g-1,\eta_L})\iso \rat_\ell$.  On the other hand, since
$\calx_{1,\eta_3}$ is a supersingular elliptic curve,
$\e(\calx_{1,\eta_3})$ is the quaternion algebra ramified only at $p$
and $\infty$, and thus $\e_\ell(\calx_{1,\eta_3}) \iso
\mat_2(\rat_\ell)$.  Therefore 
$\e(\calx_{g,\til\kappa_{g-1,1}(\eta_L\cross \eta_3)}) \iso
\e(\calx_{g-1,\eta_L})\oplus \e(\calx_{1,\eta_3})$, and $\e_\ell(\calx_{g,\til\kappa_{g-1,1}(\eta_L\cross \eta_3)})$ acts on $V_\ell(\calx_{g,s})$ as $\diag_{2g-2}(\rat_\ell) \oplus \mat_2(\rat_\ell)$. 
Similarly, $\e_\ell(\calx_{g,\til\kappa_{1,g-1}(\eta_1 \cross
  \eta_R)})$ acts as $\mat_2(\rat_\ell)\oplus
\diag_{2g-2}(\rat_\ell)$.   
Then $$\e_\ell(\calx_{g,\xi}) \subseteq \e_\ell(\calx_{g,\til\kappa_{g-1,1}(\eta_L\cross \eta_3)})
\cap \e_\ell(\calx_{g,\til\kappa_{1,g-1}( \eta_1 \cross \eta_R)})$$ 
so $\e_\ell(\calx_{g,\xi})$ acts on $V_\ell(\calx_{g,s})$ as $\diag_{2g}(\rat_\ell)$.
Thus, $\e_\ell(\calx_{g,\xi}) \iso \rat_\ell$ and $\End(\calx_{g,\xi}) \iso \integ$.
\end{proof}

%Note that we had trouble proving that $X_5$ is simple directly:
%Suppose there is an isogeny $X_5 \sim A_1 \oplus A_2$ for two abelian varieties $A_1$ and $A_2$.
%Without loss of generality, ${\rm dim}(A_2)$ equals $3$ or $4$.
%If ${\rm dim}(A_2)=3$, then there is a curve $W$ of genus $3$ such that $\jac(W) \iso A_2$.
%The inclusion of $A_2$ into $X_5$ yields a cover $\psi:\calc_{5,\xi_5} \to W$.
%The Riemann-Hurwitz formula implies that ${\rm deg}(\psi)=2$ and thus $\psi$ is Galois.
%Since the hyperelliptic involution commutes with ${\rm Gal}(\psi)$, this implies that 
%$\aut(\calc_{5,\xi_5})$ contains a Klein-four group, which contradicts the fact that   
%$\aut(\calc_{5,\xi_5}) = \st{\pm 1}$ \cite[Thm.\ 3.7]{AGP}.
%If ${\rm dim}(A_2)=4$, then GAP (note $A_2$ is not necessarily a Jacobian).
%Thus? $X_5$ is simple.

\begin{remark}
For $g\le 3$ and $1 \leq f \leq g$, it is also true that
$\End(\calx_{3,\xi}) \iso \integ$ and $\calx_{3,\xi}$ is simple for $\xi$ a geometric generic point 
of $\calh_g^f$.
More generally, for $g \geq 1$ and $f \geq 1$, the statement of 
Theorem \ref{Tendz} can be proved as an application of
Theorem \ref{Tprankhypermono} as in \cite[Application 5.7]{AP:monoprank}.
\end{remark}

\section{Monodromy} \label{Smono}

In this section, we determine the $\ell$-adic monodromy of components of $\calh_g^f$ for odd primes $\ell$.
The proof uses an inductive process which depends on the degeneration results from Section \ref{Sboundresults}.
For $f \geq 1$, the base case $g=2$ relies on a special case of \cite[Prop.\ 4.4]{chailadic}. 
When $f=0$, the base case $g=4$ relies on the results on endomorphism
rings from Section \ref{Sendresults}. 

\subsection{Integral monodromy} \label{Sdefmono}

We summarize the discussion in \cite[Sec.\ 3.1]{AP:trielliptic} about
$\integ/\ell$- and $\integ_\ell$-monodromy.  
Let $S$ be a connected $k$-scheme on which the prime $\ell$ is invertible. 
Let $\pi:C \ra S$ be a relative curve of compact type whose fibres
have genus $g$. 
Then $R^1\pi_*(\mmu_\ell)$, or equivalently $\pic^0(C)[\ell]$, is an \'etale sheaf of $\integ/\ell$-modules.
If $s$ is a geometric point of $S$, then $R^1\pi_*(\mmu_\ell)$ is equivalent to a
linear representation 
$$\rho_{C\ra S,\integ/\ell}:\pi_1(S,s) \ra \aut((R^1\pi_*(\mmu_\ell))_s) \iso \gl_{2g}(\integ/\ell).$$ 
Let $\mono_\ell(C \ra S,s)$ be the image of $\rho_{C\ra
S,\integ/\ell}$ and let $\mono_\ell(C \ra S)$ be the isomorphism 
class of this image as an abstract group.  If the family $C \ra S$ is
clear from context, these will be denoted $\mono_\ell(S,s)$ and
$\mono_\ell(S)$, respectively.  There is a canonical polarization on
$\pic^0(C)$, and thus (after a choice of $\ell^{th}$ root of unity on
$S$) there is a symplectic pairing on $\pic^0(C)[\ell]$.  Therefore, there is
an inclusion of groups $\mono_\ell(S) \subseteq
\sp_{2g}(\integ/\ell)$.  Similarly, for each natural number $n$ there
is a representation $$\rho_{C \ra S, \integ/\ell^n}: \pi_1(S,s) \ra
\aut((R^1\pi_*(\mmu_{\ell^n}))_s).$$  
Let $\mono_{\integ_\ell}(C \ra S,s) =
\invlim n \rho_{C\ra S, \integ/\ell^n}(\pi_1(S,s))$, and let
$\mono_{\rat_\ell}(C \ra S,s)$ be the Zariski closure of
$\mono_{\integ_\ell}(C \ra S,s)$ in
$\aut(\invlim n(R^1\pi_*(\mmu_{\ell^n}))_s\tensor\rat_\ell) \iso
\gl_{2g,\rat_\ell}$.  

If $C \ra S$ is a stable curve such that $C$ has compact type over an 
open dense subscheme $U \subset S$, then for each coefficient ring $\Gamma
\in \st{\integ/\ell, \integ_\ell, \rat_\ell}$ one can define $\mono_\Gamma(C \ra
S,s) = \mono_\Lambda(C\rest U \ra U, s)$.

One can employ an analogous formalism to define the monodromy group  
of a relative curve over a stack \cite{noohi}.

\subsection{Monodromy of the hyperelliptic $p$-rank strata} \label{Shyper}

In this section, let $p$ and $\ell$ be distinct odd primes.  We find
the integral monodromy of the $p$-rank strata $\calh_g^f$ when $1 \le
f \le g$.  
The integral monodromy of $\calh_g$, which is the same as
the case $f=g$, already appears in \cite[Thm.\ 3.4]{AP:trielliptic} (see also
unpublished work of J.-K. Yu, and \cite[Thm.\ 5.1]{hall06}).  

The following argument shows that to determine the monodromy of families
of hyperelliptic curves, one may work with either $\bar\calh_g$ or
$\til\calh_g$.

\begin{lemma}
\label{Laddmarking}
Let $p$ and $\ell$ be distinct odd primes and suppose $g\ge 2$.
Let $S\subset \calh_g$ be irreducible and let $\til S$ be an irreducible
component of $\bar S \cross_{\bar\calh_g} \til \calh_g$.  Then
$\mono_\ell(\til S) \iso \sp_{2g}(\integ/\ell)$ if and only if
$\mono_\ell(S) \iso \sp_{2g}(\integ/\ell)$.
\end{lemma}

\begin{proof}
Since $\til S \ra S$ is finite, $\mono_\ell(\til S)$ is a subgroup of
$\mono_\ell(\bar S)$.  
If $\mono_\ell(\til S) \iso \sp_{2g}(\integ/\ell)$ then
$\mono_\ell(\til S)$ is maximal, and thus so are 
$\mono_\ell(\bar S)$ and $\mono_\ell(S)$.

Conversely, suppose $\mono_\ell(S) \iso \sp_{2g}(\integ/\ell)$.  Since
$\varpi_g:\til \calh_g \ra \bar\calh_g$ is \'etale with Galois group
$\sym(2g+2)$, the cover $\til 
S \ra \bar S$ is Galois with Galois group $G\subseteq \sym(2g+2)$.   To show
$\mono_\ell(\til S)\iso \sp_{2g}(\integ/\ell)$, it suffices by the
argument of \cite[Lemma 3.3]{AP:trielliptic} to show
that $G$ and $\sp_{2g}(\integ/\ell)$ have no common nontrivial
quotient.  This holds since the smallest integer
$N$ for which there exists an embedding of the finite simple group
$\psp_{2g}(\integ/\ell)$ into 
$\sym(N)$ is $N = (\ell^{2g}-1)/(\ell-1) > 2g+2$ \cite[Thm.\
3]{grechkoseeva03}.
\end{proof}

\begin{theorem} \label{Tprankhypermono} 
\label{thmainprankpos}
Let $p$ and $\ell$ be distinct odd primes.  Suppose $g \geq 1$
and $1 \le f \le 
  g$. Let $S$ be an irreducible component of $\calh_g^{f}$, the
  $p$-rank $f$ stratum in $\calh_g$.  Then $\mono_\ell(S) \iso
  \sp_{2g}(\integ/\ell)$ and $\mono_{\integ_\ell}(S) \iso \sp_{2g}(\integ_\ell)$.
\end{theorem}

\begin{proof}
The proof is by induction on $g$.
The base cases involve the monodromy of
$\calh_2^2$ and $\calh_2^1$, which follow from \cite[Prop.\
4.4]{chailadic}; see \cite[Thm.\ 4.5]{AP:monoprank}.  

Now suppose $g \ge 3$ and $1 \le f \le g$.  As an inductive
hypothesis assume, for all pairs $(g',f')$ where $1 \le f' \le g' <
g$, that $\mono_\ell(S') \iso \sp_{2g'}(\integ/\ell)$ for every
irreducible component $S'$ of $\calh_{g'}^{f'}$.

Let $S$ be an irreducible component of $\calh_g^f$.  Recall the
degeneration types identified immediately before Corollary
\ref{cornewdegen11}.  If $f=g$, let $(f_1,f_2,f_3)=(1,g-2,1)$ as in case (A);
if $f=g-1$, let $(f_1,f_2,f_3)=(0,g-2,1)$ as in case (B); 
and if $1 \leq f \leq g-2$, let $(f_1,f_2,f_3)=(0,f,0)$ as in case (C).
By Corollary \ref{cornewdegen11}, 
there are irreducible components $\til S$ of $\bar S
\cross_{\bar\calh_g} \til \calh_g$,  $\til S_1$ of $\til\calh_{1}^{f_1}$,
$\til S_2$ of $\til\calh_{g-2}^{f_2}$ 
and $\til S_3$ of $\til \calh_{1}^{f_3}$;
and there are irreducible components $\til S_R$ of
$\til\calh_{g-1}^{f_2+f_3}$ and $\til S_L$ of 
$\til\calh_{g-1}^{f_1+f_2}$;
such that the restriction of the clutching maps yields a commutative diagram
\begin{equation}
\label{diagclutchhyp}
\xymatrix{
\til S_1 \cross \til S_2 \cross \til S_3 \ar[r]  \ar[d]
\ar[dr]^{\til \kappa_{1,g-2,1}}& 
 \til S_1 \cross \til S_R \ar[d]\\
 \til S_L \cross \til S_3 \ar[r] & \Delta_{1,1}[\til S].
}
\end{equation}
In all cases $f_1+f_2$ and $f_2+f_3$ are positive, and so $\til S_L$
and $\til S_R$  have monodromy $\sp_{2(g-1)}(\integ/\ell)$, by
induction and Lemma \ref{Laddmarking}. 

The rest of the proof is identical to that of \cite[Thm.\
3.4]{AP:trielliptic}.  Briefly, one calculates the monodromy group of
$\til S$ at a point $s$ in the image of $\til S_1 \cross \til S_2 \cross
\til S_3$ under $\til \kappa_{1,g-2,1}$.  On one hand, there is an a priori inclusion
$\mono_\ell(\til S,s) \subseteq \sp_{2g}(\integ/\ell)$.  On the other hand,
the previous paragraph shows that $\mono_\ell(\til S,s)$ contains two distinct subgroups isomorphic to
$\sp_{2(g-1)}(\integ/\ell)$.  A group-theoretic result shows that
$\mono_\ell(\til S,s) \iso \sp_{2g}(\integ/\ell)$.
The result then follows from Lemma \ref{Laddmarking}.

The proof that $\mono_{\integ_\ell}(S) \iso \sp_{2g}(\integ_\ell)$ is
identical.
\end{proof}

\subsubsection{A $p$-adic complement}
\label{subsecpadic}

In this section we determine the $p$-adic monodromy of components of
the $p$-rank strata $\calh_g^f$ when $f \geq 1$. 

Let $S$ be a connected scheme of characteristic $p$ with geometric
point $s$, and let $X \ra S$ be an abelian 
scheme with constant $p$-rank $f$.  The group scheme
$X[p]$ and $p$-divisible group $X[p^\infty]$ admit largest \'etale
quotients, $X[p]^\et$ and $X[p^\infty]^\et$.  These are respectively
classified by homomorphisms $\pi_1(S,s) \ra
\aut(X[p]^\et)_s \iso \gl_f(\integ/p)$ and $\pi_1(S,s) \ra
\aut(X[p^\infty]^\et)_s \iso \gl_f(\integ_p)$, whose images are denoted
$\mono_p(X \ra S)$ and $\mono_{\integ_p}(X \ra S)$, or simply
$\mono_p(S)$ and $\mono_{\integ_p}(S)$.

\begin{lemma}
\label{lemlabelpadic}
Suppose $g \ge 1$ and $1 \le f \le g$.  Let $S$ be an irreducible
component of $\calh_g^f$, and let $\til S$ be an irreducible component
of $S \cross_{\bar\calh_g} \til\calh_g$.  Then $\mono_p(S) \iso
\gl_f(\integ/p)$ if and only if $\mono_p(\til S) \iso
\gl_f(\integ/p)$.
\end{lemma}

\begin{proof}
Since $\til S \ra \bar S$ is finite, $\mono_p(\til S)$ is a subgroup of
$\mono_p(\bar S)$.  If $\mono_p(\til S) \iso \gl_{f}(\integ/p)$ then
$\mono_p(\til S)$ is maximal, and thus so are $\mono_p(\bar S)$ and $\mono_p(S)$.

Conversely, suppose $\mono_p(S) \iso \gl_f(\integ/p)$.  
Let $S^*$ be the closure of $S$ in $\bar\calh_g - \Delta_0[\bar\calh_g]$ and 
let $\til S^*$ be the closure of $\til S$ in $\til\calh_g - \Delta_0[\til \calh_g]$.  
Let $T = S^* - (S^*)^f$ be the locus with $p$-rank smaller than $f$.
Then $T$ is nonempty by Corollary \ref{cordegenprank}.
On one hand, $\til S^* \ra S^*$ is \'etale since $p$ is odd and the cover $\til S^* \ra S^*$ is tantamount to a partial
level-two structure.
On the other hand, the $\gl_f(\integ/p)$-cover $H_f :=
\hom_S((\integ/p), \jac(\calc_{g,S})[p]^\et) \ra S$ is ramified along
$T$.  Therefore, the covers $H_f \ra S$ and $\til S \ra S$ are
disjoint, and $\mono_p(\til S) = \mono_p(S) \iso \gl_f(\integ/p)$.
\end{proof}

\begin{proposition}
  \label{proppadic}
Suppose $g\ge 2$ and $1 \le f \le g$.  Let $S$ be an irreducible
component of $\calh_g^f$.  Then $\mono_p(S) \iso \gl_f(\integ/p)$ and
$\mono_{\integ_p}(S) \iso \gl_f(\integ_p)$.
\end{proposition}

\begin{proof}
First suppose $f=g$.  When $g=2$, the result for $\calh_2^2$, or equivalently $\calm_2^2$,
is a special case of \cite[Thm.\ 2.1]{ekedahlmono}.
For $g\ge 3$, suppose as an inductive hypothesis that 
$\mono_p(\calh_{g-1}^{g-1}) \iso \gl_{g-1}(\integ/p)$.
By Lemma \ref{lemlabelpadic}, $\mono_p(\til\calh_{g-1}^{g-1}) \iso
\gl_{g-1}(\integ/p)$.  

Recall the diagram \eqref{diagclutchbox}, and 
consider a geometric point $s \in \til \calh_g^g$ in the image of $\til\calh_1^1\cross \til \calh_{g-2}^{g-2} \cross \til
\calh_1^1$ under $\til \kappa_{1,g-2,1}$.    
By the inductive hypothesis and Lemma \ref{lemlabelpadic}, 
$\mono_p(\til\calh_g^g,s)$ contains two distinct copies of
$\gl_{g-1}(\integ/p)$ and thus  
equals $\gl_g(\integ/p)$ by the argument of \cite[Lemma 3.2]{AP:trielliptic}.

Now suppose $1 \leq f \leq g-1$.
By Corollary \ref{cornewdegen1}(a), there are irreducible
components $\til V_1 \subset \til \calh_1^0$ and $\til V_2 \subset \til
\calh_{g-1}^f$ such that $\bar S$ contains $\til\kappa_{1,g-1}(\til V_1
\cross \til V_2)$.  By the inductive hypothesis and Lemma
\ref{lemlabelpadic}, $\mono_p(\til V_2) \iso \gl_f(\integ/p)$, and thus
$\mono_p(S) \iso \gl_f(\integ/p)$ as well.

The proof that $\mono_{\integ_p}(S) \iso \gl_f(\integ_p)$ is identical. 
\end{proof}

\subsection{$p$-rank zero: monodromy}
\label{subsecprzmono}
In this section, we determine the integral monodromy of components of $\calh_g^0$
under a few mild hypotheses.
The monodromy group of $\calh_2^0$ is small, since
supersingular families of abelian varieties have finite $\ell$-adic
monodromy groups; and the methods of \cite{chailadic} do not apply to
$\calh_g^0$ for $g\ge 3$, because the hyperelliptic Torelli locus is
not a Hecke-stable subset of $\cala_g$.
Thus our proof requires another base case when $f=0$.  For lack of a strategy to 
calculate the $\ell$-adic monodromy of $\calh_3^0$, we analyze
the case when $g=4$ and $f=0$ using results on endomorphism rings from Section \ref{Sendresults}.
We thus determine the mod-$\ell$ monodromy group of components of 
$\calh_g^0$ when $g \geq 4$, for all but finitely many $\ell$.
Note that in Theorem \ref{Tprankzerohypermono}, the set of
exceptional primes depends on the characteristic $p$ of the base
field, but not on $g$, so that our results are valid for $\ell \gg_p
0$.

\begin{lemma}
\label{lemmonoh30}
Let $S$ be an irreducible component of $\calh_3^0$.
\begin{alphabetize}
\item Either $\mono_{\rat_\ell}(S) \iso \sp_{6, \rat_\ell}$ for
  all $\ell\not =p$, or there exists a totally real field $L$ such that
  $\mono_{\rat_\ell}(S) \iso (\res_{L/\rat}\SL_2) \cross \rat_\ell$
  for all $\ell\not = p$.
\item Let $s \in S$ be a geometric point.  For $\ell$ in a set of
  positive density, there exists a torus $T_\ell \subset
  \mono_{\rat_\ell}(S,s)$ which acts irreducibly on $V_\ell(\calx_s)$.
\end{alphabetize}
\end{lemma}

\begin{proof}
  Let $\eta$ be the generic point of $S$, and consider the dichotomy
  of Lemma \ref{lemqorcubic}.  If $\End(\calx_{3,\eta}) \iso \integ$,
  then one knows (e.g., \cite[Thm.\ 3]{serrevigneras}) that
  $\mono_{\rat_\ell}(S) \iso \sp_{2g,\rat_\ell}$ for all $\ell\not
  =p$.

Otherwise, if $\e(\calx_{3,\eta}) \iso L$, a totally real cubic field, then $S$
coincides with (a component of) the $p$-rank zero locus of a Hilbert
modular threefold attached to $L$.  Therefore, $\mono_{\rat_\ell}(S)
\iso \res_{L/\rat}\SL_2 \cross_\rat \rat_\ell$ for all $\ell\not = p$
\cite[Lemma 6.5]{yucrelle09}.  This proves (a).

For (b), if $\mono_{\rat_\ell}(S)$ is the symplectic group, for each
prime $\ell$ choose a CM field $K(\ell)$ of degree $6$ which is inert
at $\ell$.

If instead $\e(\calx_{3,\eta})$ is isomorphic to a totally real cubic
field $L$, then for $\ell$ in a set of positive density, $\ell$ remains
inert in $L$.  (If $L$ is Galois over $\rat$, this is clear from the
Chebotarev theorem.   Otherwise, the Galois closure $\til L$ of $L$ has
$\gal(\til L/\rat) \iso \sym(3)$; for $\ell$ in a set of positive
density, $\ell$ splits into two primes in $\til L$, and thus $\ell$ is
inert in $L$.)  For each such $\ell$, let $a_\ell \in \rat$ be a
positive number such that $-a_\ell$ is not a square $\bmod \ell$; then
$K(\ell) := L[\sqrt{-a_\ell}]$ is a totally imaginary field which is
inert at $\ell$. 

Then the norm one torus $T_\ell :=
(\res_{K(\ell)/\rat}\gp_m)^{(1)}\cross_\rat \rat_\ell$ is a suitable torus.
\end{proof}

\begin{proposition}
\label{propmonoh40}
Let $S$ be an irreducible component of $\calh_4^0$.  For each
$\ell\not = p$, $\mono_{\rat_\ell}(S) \iso  \sp_{8, \rat_\ell}$.
\end{proposition}

\begin{proof}
Fix a geometric point $s \in S$.
By \cite[Thm.\ 3.3]{larsenpinkabvar},
there exists a connected reductive group $G/\rat$ and an
$8$-dimensional representation $V$ of $G$ such that for $\ell\gg 0$,
the representation $\mono_{\rat_\ell}(S,s) \ra 
\aut(V_\ell(\calx_{4,s}))$ is
isomorphic to the representation $G\cross_\rat \rat_\ell \ra
\aut(V\tensor_\rat\rat_\ell)$.  From Zarhin's theorem and the
classification of semisimple Lie algebras (see, e.g., \cite[Lemma
1.3]{noot00}), either $G = \sp_8$, or the representation is of
Mumford type \cite{mumfordnote}; and in each case,
$\mono_{\rat_\ell}(S)$ is in fact isomorphic to $G\cross_\rat \rat_\ell$ for all $\ell\not = p$.
If the representation is of Mumford type, then $G$ is isogenous
to a twist of $\res_{K/\rat}\SL_2$ for some totally real cubic field $K$, and in
particular has dimension nine.  Therefore, to prove the claim, it
suffices to show that $\dim_{\rat_\ell}\mono_{\rat_\ell}(S,s) \geq 10$.

By Corollary \ref{cornewdegen11}, $\bar S$ intersects $\Delta_{1,1}(\bar \calh_{1,1})$.  As in the proof of
Theorem \ref{thmainprankpos}, one can compute $\mono_{\rat_\ell}(S,s)$ at a point
$s$ in $\Delta_{1,1}(\bar S)$.  Then there are components $\til S_L$
and $\til S_R$ of $\til \calh_3^0$, and components $\til S_1$ and
$\til S_3$ of $\til \calh_1^0$, such that $\mono_{\rat_\ell}(S,s)$
contains distinct subgroups $\mono_{\rat_\ell}(\til S_L)
\cross \mono_{\rat_\ell}(\til S_3)$ and $\mono_{\rat_\ell}(\til S_1)
\cross \mono_{\rat_\ell}(\til S_R)$.  Moreover, by Lemma
\ref{lemmonoh30}(a), each of $\mono_{\rat_\ell}(\til
S_L)$ and $\mono_{\rat_\ell}(\til S_R)$ has dimension either $21$ or $9$.
Therefore, $\dim_{\rat_\ell} \mono_{\rat_\ell}(S,s) \geq 10$, and
$\mono_{\rat_\ell}(S,s) \iso \sp_{8, \rat_\ell}$ for all $\ell\not =p$.
\end{proof}

\begin{theorem}
\label{Tprankzerohypermono}
If $\ell \gg_p 0$, if $g \ge 4$ and if $S$ is an
irreducible component of $\calh_g^0$, then $\mono_\ell(S) \iso
\sp_{2g}(\integ/\ell)$ and $\mono_{\integ_\ell}(S)\iso \sp_{2g}(\integ_\ell)$.
\end{theorem}

\begin{proof}
If $g =4$ and $S$ is an irreducible component of $\calh_g^0$, then 
Proposition \ref{propmonoh40} provides the hypothesis needed to prove 
$\mono_{\integ_\ell}(S)\iso \sp_{8}(\integ_\ell)$ for all but finitely many $\ell$, e.g., \cite[8.2]{serrevigneras}.
This yields the result for $g=4$ since $\calh_4^0$ has only finitely many irreducible components.
For $g >4$, the proof is identical to that of Theorem \ref{thmainprankpos}, 
with Corollary \ref{cornewdegen11} being used to degenerate to a component of $\til
\kappa_{1,g-2,1}(\til\calh_1^0 \cross \til\calh_{g-2}^0 
\cross \til\calh_1^0)$.
\end{proof}

\begin{remark}
The assertion of Theorem \ref{Tprankzerohypermono} is false for
$\calh^0_1$ and $\calh_2^0$ if $\ell \ge 5$.  Indeed, a
hyperelliptic curve of genus $g\le 2$ and $p$-rank $0$ is supersingular. 
Since a supersingular $p$-divisible group over a scheme $S$ becomes trivial
after a finite pullback $\til S \ra S$, the monodromy group
$\mono_{\integ_\ell}(\calh^0_g)$ is finite for $g \leq 2$.   The
ambiguity in Lemma \ref{lemqorcubic} propagates to Lemma
\ref{lemmonoh30}, and we do not
know whether the assertion of Theorem \ref{Tprankzerohypermono} is
true for $\calh^0_3$.
\end{remark}

\subsection{Arithmetic applications}
\label{subsecapp}

The results of the previous section about the monodromy of components of $\calh_g^f$ have arithmetic applications involving hyperelliptic curves over finite fields.
For example, they imply that there exist hyperelliptic curves of given genus and $p$-rank with absolutely simple Jacobian (Application \ref{appabssimp}).  
Moreover, they give estimates for the
proportion of hyperelliptic curves with a given genus and $p$-rank 
which have a rational point of order $\ell$ on the
Jacobian (Application  \ref{appclass}) or for which the numerator of the
zeta function has large splitting field (Application \ref{appzeta}).

Throughout this section, $\ff$ denotes a finite extension of $\ff_p$.

\subsubsection{Technical context} 

We do not include proofs in this section, since they are very similar to those found in \cite[Section 5]{AP:monoprank}.
Here is a brief description of the main ideas involved.  One first
defines $\calh_g$ over the category of $\ff_p$-schemes and defines the
arithmetic monodromy group of a substack of $\calh_g$.  For a relative
curve $\pi: C \to S/\ff$ of genus $g \ge 2$ defined over a finite field,
one shows that if $\mono^\geom_\ell(S) \iso \sp_{2g}(\integ/\ell)$,
then $\mono^\geom_{\integ_\ell}(S) \iso \sp_{2g}(\integ_\ell)$;
and $\mono^{\rm arith}_{\integ_\ell}(S)$ has finite index in
$\gsp_{2g}(\integ_\ell)$; and $\mono^\geom_{\rat_\ell}(S) \iso
\sp_{2g,\rat_\ell}$ \cite[Lemma 5.1]{AP:monoprank}.

Secondly, in order to use Chebotarev arguments for curves over finite
fields, it is necessary to add rigidifying data, such as the data of a
tricanonical structure, so that the corresponding moduli
problems are representable by schemes.  Recall that $\Omega^{\tensor
3}_{C/S}$ is very ample, that $\pi_*(\Omega_{C/S}^{\tensor 3})$ is a
locally free $\calo_S$-module of rank $5g-5$, and that sections of
this bundle define a closed embedding $C\inject \proj^{5g-5}_S$.  A
tricanonical ($3K$) structure on $\pi: C \ra S$ is a choice of isomorphism
$\calo_S^{\oplus 5g-5} \iso
\pi_*(\Omega_{C/S}^{\tensor 3})$, and 
the only automorphisms of a hyperelliptic curve with $3K$-structure are the identity and the hyperelliptic involution.
The moduli space $\calh_{g,3K}$ of smooth hyperelliptic curves of genus $g$ equipped with a $3K$-structure
is representable by a scheme \cite[10.6.5]{katzsarnak},
\cite[Prop. 5.1]{mumfordgit}.  

Third, since $\calh_g$ may be constructed as the quotient of $\calh_{g,3K}$ by $\gl_{5g-5}$,
the forgetful functor $\psi_g:\calh_{g,3K} \ra \calh_g$ is open \cite[p.\ 6]{mumfordgit} and a
fibration with connected fibers \cite[Thm.\ A.12]{noohi}. 
Thus, if $S\subset \calh_g$ is a connected substack and
$S_{3K} = S \cross_{\calh_g} \calh_{g,3K}$, then $\mono_{\ell}(S_{3K}) \iso \mono_{\ell}(S)$ 
\cite[Lemma 5.2]{AP:monoprank}.
Since the data of a tricanonical structure exists
Zariski-locally on the base, one can relate point
counts on $\calh_{g,3K}(\ff)$ to those on $\calh_g(\ff)$.
Specifically, if $s \in \calh_g^f(\ff)$ is such that ${\rm Aut}(\calc_{g,s}) \iso \{\pm 1\}$, 
then the fiber of $\calh_{g,3K}(\ff)$ over $s$ consists of $\abs{\gl_{2g}(\ff)}/2$ points 
 \cite[10.6.8]{katzsarnak}.

\subsubsection{Application to simple Jacobians}

Using the $\rat_\ell$-monodromy of $\calh_g^f$, we deduce that there 
exist hyperelliptic curves of genus $g$ and $p$-rank $f$ with absolutely simple Jacobian. 

\begin{application}
\label{appabssimp}
Suppose $g \ge 1$ and $0 \leq f \leq g$ with $f \not = 0$ if $g \leq 2$.  
Let $S$ be an irreducible component of $\calh_g^f$.
Then there exists $s\in S(\bar\ff)$ such that the Jacobian of $\calc_{g,s}$ is absolutely simple.
\end{application}

\begin{proof}
If $f \geq 1$ or $g \geq 4$, then this follows from Theorems \ref{Tprankhypermono}
and \ref{Tprankzerohypermono} using \cite[Prop. 4]{chaioort01}.
If $f=0$ and $g =3$, then this follows from Lemma
\ref{lemmonoh30} using \cite[Rem.\ 5(i)]{chaioort01}.
\end{proof}

\begin{remark}
For arbitrary $g$, it is unknown how to deduce Application \ref{appabssimp} from Theorem \ref{Tendz}.
On an unrelated note, under the hypotheses of Application \ref{appabssimp},
one can deduce that $\aut(\calc_{g,\eta})= \st{\pm 1}$, 
which yields a new proof of \cite[Thm.\ 3.7]{AGP}.
\end{remark}
 
\subsubsection{Application to class groups}

Recall that if $s \in \calh_g(\ff)$,
then $\pic^0(\calc_{g,s})(\ff)$ is isomorphic to the class group of the function field $\ff(\calc_{g,s})$.
The size of the class group is divisible by $\ell$ exactly when there is a point of order $\ell$ on the 
Jacobian.
Roughly speaking, Application \ref{appclass}
shows that among all curves over $\ff$ of specified genus and $p$-rank,
slightly more than $1/\ell$ of them have an $\ff$-rational point of
order $\ell$ on their Jacobian.  

\begin{application} \label{appclass}
Suppose  $\ell$ and $p$ are distinct odd primes, $g \ge 1$ and $1 
\leq f \leq g$. 
Suppose $S$ is an irreducible component of $\calh_g^f$ such that $S(\ff) \not = \emptyset$.
Let $m$ be the image of $\abs \ff$ in $(\integ/\ell)\units$.  
There exists a rational function $\alpha_{g,m}(T) \in \rat(T)$ 
and a constant $B= B(p,g,\ell)$ such that
\begin{equation} 
\label{eqclass}
\abs{\frac{\#\st{ s\in S(\ff) : \ell | \#\pic^0(\calc_{g,s})(\ff)}}{\#
    S(\ff)}  - \alpha_{g,m}(\ell)} < \frac B{\sqrt q}.
\end{equation}
If $f=0$ and $g \ge 4$, the same result is true for all $\ell \gg_p 0$.
\end{application}

\begin{proof}
The proof is very similar to that of \cite[Application
5.9]{AP:monoprank} and uses Theorems \ref{Tprankhypermono}/\ref{Tprankzerohypermono} and
\cite[Thm.\ 9.7.13]{katzsarnak}.
\end{proof}

\begin{remark}
\label{remachtercl} For $\ell$ odd, one knows that $\alpha_{g,1}(\ell) =
\frac{\ell}{\ell^2-1} + \calo(1/\ell^3)$, 
while 
$\alpha_{g,m}(\ell) = \oneover{\ell-1}+\calo(1/\ell^3)$ if $m \not
= 1$.  A formula for $\alpha_{g,1}(\ell)$ is given in
\cite{achtercl}.
\end{remark}

\subsubsection{Application to zeta functions}

If $C/\ff$ is a smooth projective curve of genus $g$, its zeta function has the form
$L_{C/\ff}(T)/(1-T)(1-qT)$, where $L_{C/\ff}(T) \in \integ[T]$ is a polynomial
of degree $2g$.  The principal polarization on the Jacobian of $C$
forces a symmetry among the roots of $L_{C/\ff}(T)$; the largest
possible Galois group for the splitting field over $\rat$ of
$L_{C/\ff}(T)$ is the Weyl group of $\sp_{2g}$ which is a group of
size $g!2^g$.

\begin{application}
\label{appzeta}
Suppose $g \ge 1$ and $1 \leq f \leq g$, or that $g \ge 4$
and $f = 0$.  Suppose $p > 2g+1$ and that $S$ is an irreducible component of $\calh_g^f$ such that $S(\ff) \not = \emptyset$.
There exists a constant $\gamma= \gamma(g)>0$ and a constant $E =
E(p,g)$ such that
\begin{equation}
\label{eqzeta}
\frac{\# \st{ s \in S(\ff) : L_{\calc_{g,s}/\ff}(T)\text{ is reducible, or has
      splitting field with degree }< 2^gg!}}{\#S(\ff)} < E q^{-\gamma}.
\end{equation}
\end{application}

\begin{proof}
The proof is very similar to that of \cite[Application 5.11]{AP:monoprank}
and uses Theorems \ref{Tprankhypermono}/\ref{Tprankzerohypermono} 
and \cite[Thm. 6.1 and Remark 3.2.(4)]{kowalskisieve}.
\end{proof}

\bibliographystyle{amsalpha}
\bibliography{jda}

\def\cprime{$'$}
\providecommand{\bysame}{\leavevmode\hbox to3em{\hrulefill}\thinspace}
\providecommand{\MR}{\relax\ifhmode\unskip\space\fi MR }
\providecommand{\MRhref}[2]{%
  \href{http://www.ams.org/mathscinet-getitem?mr=#1}{#2}
}
\providecommand{\href}[2]{#2}
\begin{thebibliography}{EvdG09}

\bibitem[Ach06]{achtercl}
Jeffrey~D. Achter, \emph{The distribution of class groups of function fields},
  J. Pure Appl. Algebra \textbf{204} (2006), no.~2, 316--333. \MR{MR2184814
  (2006h:11132)}

\bibitem[AGP08]{AGP}
Jeffrey~D. Achter, Darren Glass, and Rachel Pries, \emph{Curves of given
  $p$-rank with trivial automorphism group}, Mich. Math. J. \textbf{56} (2008),
  no.~3, 583--592.

\bibitem[AP07]{AP:trielliptic}
Jeffrey~D. Achter and Rachel Pries, \emph{The integral monodromy of
  hyperelliptic and trielliptic curves}, Math. Ann. \textbf{338} (2007), no.~1,
  187--206.

\bibitem[AP08]{AP:monoprank}
\bysame, \emph{Monodromy of the {$p$}-rank strata of the moduli space of
  curves}, Int. Math. Res. Not. IMRN (2008), no.~15, Art. ID rnn053, 25.
  \MR{MR2438069 (2009i:14030)}

\bibitem[BLR90]{blr}
Siegfried Bosch, Werner L{\"u}tkebohmert, and Michel Raynaud, \emph{N\'eron
  models}, Ergebnisse der Mathematik und ihrer Grenzgebiete (3) [Results in
  Mathematics and Related Areas (3)], vol.~21, Springer-Verlag, Berlin, 1990.
  \MR{MR1045822 (91i:14034)}

\bibitem[CH88]{cornalbaharris}
Maurizio Cornalba and Joe Harris, \emph{Divisor classes associated to families
  of stable varieties, with applications to the moduli space of curves}, Ann.
  Sci. \'Ecole Norm. Sup. (4) \textbf{21} (1988), no.~3, 455--475. \MR{MR974412
  (89j:14019)}

\bibitem[Cha05]{chailadic}
Ching-Li Chai, \emph{Monodromy of {H}ecke-invariant subvarieties}, Pure Appl.
  Math. Q. \textbf{1} (2005), no.~2, 291--303. \MR{MR2194726 (2006m:11084)}

\bibitem[CO01]{chaioort01}
Ching-Li Chai and Frans Oort, \emph{A note on the existence of absolutely
  simple {J}acobians}, J. Pure Appl. Algebra \textbf{155} (2001), no.~2-3,
  115--120. \MR{MR1801409 (2002a:14020)}

\bibitem[Dem72]{demazure}
Michel Demazure, \emph{Lectures on $p$-divisible groups}, Springer-Verlag,
  Berlin, 1972, Lecture Notes in Mathematics, Vol. 302. \MR{49 \#9000}

\bibitem[DM69]{delignemumford}
P.~Deligne and D.~Mumford, \emph{The irreducibility of the space of curves of
  given genus}, Inst. Hautes \'Etudes Sci. Publ. Math. (1969), no.~36, 75--109.
  \MR{41 \#6850}

\bibitem[Eke91]{ekedahlmono}
Torsten Ekedahl, \emph{The action of monodromy on torsion points of
  {J}acobians}, Arithmetic algebraic geometry (Texel, 1989), Birkh\"auser
  Boston, Boston, MA, 1991, pp.~41--49. \MR{92g:14017}

\bibitem[Eke95]{ekedahlhurwitz}
\bysame, \emph{Boundary behaviour of {H}urwitz schemes}, The moduli space of
  curves (Texel Island, 1994), Progr. Math., vol. 129, Birkh\"auser Boston,
  Boston, MA, 1995, pp.~173--198. \MR{MR1363057 (96m:14030)}

\bibitem[EvdG09]{evdg}
Torsten Ekedahl and Gerard van~der Geer, \emph{{Cycle Classes of the EO
  Stratification on the Moduli of Abelian Varieties}}, Algebra, arithmetic and
  geometry -- {M}anin {F}estschrift (Y.~Tschinkel and Y.~Zarhin, eds.),
  Progress in Mathematics, vol. 269, Birkh\"auser, 2009, to appear.

\bibitem[FC90]{faltingschai}
Gerd Faltings and Ching-Li Chai, \emph{Degeneration of abelian varieties},
  Springer-Verlag, Berlin, 1990, With an appendix by David Mumford.
  \MR{92d:14036}

\bibitem[FvdG04]{FVdG:complete}
Carel Faber and Gerard van~der Geer, \emph{Complete subvarieties of moduli
  spaces and the {P}rym map}, J. Reine Angew. Math. \textbf{573} (2004),
  117--137. \MR{MR2084584 (2005g:14054)}

\bibitem[GP05]{GP:05}
Darren Glass and Rachel Pries, \emph{Hyperelliptic curves with prescribed
  {$p$}-torsion}, Manuscripta Math. \textbf{117} (2005), no.~3, 299--317.
  \MR{MR2154252 (2006e:14039)}

\bibitem[Gre03]{grechkoseeva03}
Maria~A. Grechkoseeva, \emph{On minimal permutation representations of
  classical simple groups}, Sibirsk. Mat. Zh. \textbf{44} (2003), no.~3,
  560--586. \MR{MR1984704 (2004b:20021)}

\bibitem[Hal06]{hall06}
Chris Hall, \emph{{Big symplectic or orthogonal monodromy modulo $\ell$}},
  August 2006, arXiv:math.NT/0608718.

\bibitem[Kat79]{katzsf}
Nicholas~M. Katz, \emph{Slope filtration of ${F}$-crystals}, Journ\'ees de
  G\'eom\'etrie Alg\'ebrique de Rennes (Rennes, 1978), Vol. I, Soc. Math.
  France, Paris, 1979, pp.~113--163. \MR{81i:14014}

\bibitem[Knu83]{knudsen2}
Finn~F. Knudsen, \emph{The projectivity of the moduli space of stable curves.
  {II}. {T}he stacks {$M\sb{g,n}$}}, Math. Scand. \textbf{52} (1983), no.~2,
  161--199. \MR{MR702953 (85d:14038a)}

\bibitem[Kow06]{kowalskisieve}
E.~Kowalski, \emph{The large sieve, monodromy and zeta functions of curves}, J.
  Reine Angew. Math. \textbf{601} (2006), 29--69. \MR{MR2289204}

\bibitem[KS99]{katzsarnak}
Nicholas~M. Katz and Peter Sarnak, \emph{Random matrices, {F}robenius
  eigenvalues, and monodromy}, American Mathematical Society, Providence, RI,
  1999. \MR{2000b:11070}

\bibitem[LP95]{larsenpinkabvar}
M.~Larsen and R.~Pink, \emph{Abelian varieties, {$l$}-adic representations, and
  {$l$}-independence}, Math. Ann. \textbf{302} (1995), no.~3, 561--579.
  \MR{MR1339927 (97e:14057)}

\bibitem[MFK94]{mumfordgit}
David Mumford, John Fogarty, and Frances Kirwan, \emph{Geometric invariant
  theory}, Ergebnisse der Mathematik und ihrer Grenzgebiete (2) [Results in
  Mathematics and Related Areas (2)], vol.~34, Springer-Verlag, Berlin, 1994.
  \MR{MR1304906 (95m:14012)}

\bibitem[Mum69]{mumfordnote}
D.~Mumford, \emph{A note of {S}himura's paper ``{D}iscontinuous groups and
  abelian varieties''}, Math. Ann. \textbf{181} (1969), 345--351. \MR{MR0248146
  (40 \#1400)}

\bibitem[Noo00]{noot00}
Rutger Noot, \emph{Abelian varieties with {$l$}-adic {G}alois representation of
  {M}umford's type}, J. Reine Angew. Math. \textbf{519} (2000), 155--169.
  \MR{MR1739726 (2001k:11112)}

\bibitem[Noo04]{noohi}
B.~Noohi, \emph{Fundamental groups of algebraic stacks}, J. Inst. Math. Jussieu
  \textbf{3} (2004), no.~1, 69--103. \MR{MR2036598 (2004k:14003)}

\bibitem[Oor74]{O:purity}
Frans Oort, \emph{Subvarieties of moduli spaces}, Invent. Math. \textbf{24}
  (1974), 95--119. \MR{54 \#12771}

\bibitem[Oor88]{oortendabvar}
\bysame, \emph{Endomorphism algebras of abelian varieties}, Algebraic geometry
  and commutative algebra, Vol.\ II, Kinokuniya, Tokyo, 1988, pp.~469--502.
  \MR{MR977774 (90j:11049)}

\bibitem[Oor91]{oorthess}
\bysame, \emph{Hyperelliptic supersingular curves}, Arithmetic algebraic
  geometry (Texel, 1989), Progr. Math., vol.~89, Birkh\"auser Boston, Boston,
  MA, 1991, pp.~247--284. \MR{MR1085262 (92c:14043)}

\bibitem[OS80]{oortsteenbrink}
Frans Oort and Joseph Steenbrink, \emph{The local {T}orelli problem for
  algebraic curves}, Journ\'ees de G\'eometrie Alg\'ebrique d'Angers, Juillet
  1979/Algebraic Geometry, Angers, 1979, Sijthoff \& Noordhoff, Alphen aan den
  Rijn, 1980, pp.~157--204. \MR{MR605341 (82i:14014)}

\bibitem[Ser00]{serrevigneras}
Jean-Pierre Serre, \emph{Lettre \`a {M}arie-{F}rance {V}igneras}, {\OE}uvres.
  {C}ollected papers. {IV}, Springer-Verlag, Berlin, 2000, 1985--1998,
  pp.~38--55. \MR{MR1730973 (2001e:01037)}

\bibitem[SZ02]{Zhu:nohypss}
Jasper Scholten and Hui~June Zhu, \emph{Hyperelliptic curves in characteristic
  2}, Int. Math. Res. Not. (2002), no.~17, 905--917. \MR{MR1899907
  (2003d:11089)}

\bibitem[Vis89]{V:stack}
Angelo Vistoli, \emph{Intersection theory on algebraic stacks and on their
  moduli spaces}, Invent. Math. \textbf{97} (1989), no.~3, 613--670.
  \MR{MR1005008 (90k:14004)}

\bibitem[Wed99]{wedhorn}
Torsten Wedhorn, \emph{Ordinariness in good reductions of {S}himura varieties
  of {P}{E}{L}-type}, Ann. Sci. \'Ecole Norm. Sup. (4) \textbf{32} (1999),
  no.~5, 575--618. \MR{2000g:11054}

\bibitem[Yam04]{yamaki04}
Kazuhiko Yamaki, \emph{Cornalba-{H}arris equality for semistable hyperelliptic
  curves in positive characteristic}, Asian J. Math. \textbf{8} (2004), no.~3,
  409--426. \MR{MR2129243 (2005j:14039)}

\bibitem[Yu09]{yucrelle09}
Chia-Fu Yu, \emph{Irreducibility of the {H}ilbert-{B}lumenthal moduli spaces
  with parahoric level structure}, J. Reine Angew. Math. \textbf{2009} (2009),
  no.~635, 187--211.

\end{thebibliography}
\end{document}